\documentclass{amsart}

\usepackage{amssymb,amsmath,amscd,amsthm}
\usepackage[OT2,T1]{fontenc}
\newcommand\textcyr[1]{{\fontencoding{OT2}\fontfamily{wncyr}\selectfont #1}}

\numberwithin{equation}{section}

\date{\today}

\newcommand{\Z}{{\mathbb Z}}
\newcommand{\R}{{\mathbb R}}
\newcommand{\C}{{\mathbb C}}

\newcommand{\T}{{\mathbb T}}

\newcommand{\D}{{\mathbb D}}

\newcommand{\PP}{{\mathbb P}}

\newtheorem{theorem}{Theorem}[section]
\newtheorem{lemma}[theorem]{Lemma}
\newtheorem{prop}[theorem]{Proposition}
\newtheorem{coro}[theorem]{Corollary}

\newtheorem*{klc}{Kotani-Last Conjecture}

\theoremstyle{definition}
\newtheorem{remark}[theorem]{Remark}
\theoremstyle{definition}
\newtheorem{defi}[theorem]{Definition}

\renewcommand{\Im}{\mathrm{Im} \, }
\renewcommand{\Re}{\mathrm{Re} \, }

\begin{document}

\title[Counterexamples to the Kotani-Last Conjecture]{Counterexamples to the Kotani-Last Conjecture for Continuum Schr\"odinger Operators via Character-Automorphic Hardy Spaces}

\author[D.\ Damanik]{David Damanik}

\address{Department of Mathematics, Rice University, Houston, TX~77005, USA}

\email{damanik@rice.edu}

\thanks{D.\ D.\ was supported in part by NSF grant DMS--1067988.}

\author[P.\ Yuditskii]{Peter Yuditskii}

\address{Abteilung f\"ur Dynamische Systeme und Approximationstheorie, Johannes Kepler Universit\"at Linz, A-4040 Linz, Austria}

\email{Petro.Yudytskiy@jku.at}

\thanks{P.\ Yu.\ was supported by the Austrian Science Fund FWF, project no: P25591-N25.
}

\begin{abstract}
The Kotani-Last conjecture states that every ergodic operator in one space dimension with non-empty absolutely continuous spectrum must have almost periodic coefficients. This statement makes sense in a variety of settings; for example, discrete Schr\"odinger operators, Jacobi matrices, CMV matrices, and continuum Schr\"odinger operators.

In the main body of this paper we show how to construct counterexamples to the Kotani-Last conjecture for continuum Schr\"odinger operators by adapting the approach developed by Volberg and Yuditskii to construct counterexamples to the Kotani-Last conjecture for Jacobi matrices. This approach relates the reflectionless operators associated with the prescribed spectrum to a family of character-automorphic Hardy spaces and then relates the shift action on the space of operators to the resulting action on the associated characters. The key to our approach is an explicit correspondence between the space of continuum reflectionless Schr\"odinger operators associated with a given set and the space of reflectionless Jacobi matrices associated with a derived set. Once this correspondence is established we can rely to a large extent on the previous work of Volberg and Yuditskii to produce the resulting action on the space of characters. We analyze this action and identify situations where we can observe absolute continuity without almost periodicity.

In the appendix we show how to implement this strategy and obtain analogous results for extended CMV matrices.
\end{abstract}

\dedicatory{\textcyr{...ne stranen kto zh?\\
 Griboedov A. S. ``Gore ot uma''}}

\maketitle

\section{Introduction}\label{s.1}

\subsection{Motivation and Background}

One-dimensional continuum Schr\"odinger operators with ergodic potentials are generated as follows. Suppose $(\Omega, \mathcal{F}, \mu)$ is a probability space, $\{ \mathcal{T}_x \}_{x \in \R}$ is a $\mu$-ergodic one-parameter group of transformations of $\Omega$, and $f : \Omega \to \R$ is measurable and (for simplicity) bounded, we consider potentials
$$
q_\omega(x) = f(\mathcal{T}_x \omega)
$$
and Schr\"odinger operators
$$
[L_{q_\omega} y](x) = - y''(x) + q_\omega(x) y(x)
$$
in $L^2(\R)$.

This class of potentials contains many important special cases, such as periodic potentials, almost periodic potentials, and random potentials. Moreover, there are several general results that hold for all operator families $\{ L_{q_\omega} \}_{\omega \in \Omega}$ generated in this way. For example, the spectrum and the spectral type are almost surely constant, that is, there are sets $\Sigma, \Sigma_\mathrm{ac}, \Sigma_\mathrm{sc}, \Sigma_\mathrm{pp}$ and a set $\Omega_0 \subseteq \Omega$ with $\mu(\Omega_0) = 1$ such that $\sigma(L_{q_\omega}) = \Sigma$ and $\sigma_\bullet(L_{q_\omega}) = \Sigma_\bullet$, $\bullet \in \{ \mathrm{ac}, \mathrm{sc}, \mathrm{pp} \}$, for every $\omega \in \Omega_0$.

It is well known that all periodic potentials lead to $\Sigma_\mathrm{sc} = \Sigma_\mathrm{pp} = \emptyset$, whereas all random potentials lead to $\Sigma_\mathrm{ac} = \Sigma_\mathrm{sc} = \emptyset$. Moreover, the existing results make the following tendency evident: the stronger the disorder in the potentials, the more singular the spectral measures of $H_\omega$ become. That is, as one moves away from the periodic (= completely ordered) case, one should expect that one eventually transitions into purely singular spectrum, that is, $\Sigma_\mathrm{ac} = \emptyset$.

There is indeed absolute continuity beyond periodicity. For example, purely absolutely continuous spectrum can be observed for certain classes of limit-periodic potentials \cite{AS81, C81, PT88}, that is, potentials that may be uniformly approximated by periodic potentials. This fits nicely with the overall picture, as these potentials are the closest a potential can be to being periodic without actually being periodic. Going a bit further, purely absolutely continuous spectrum can also be shown for quasi-periodic potentials with Diophantine frequency vectors and small analytic sampling functions; compare, for example, \cite{DG14, DS75, E92}.

Both limit-periodic potentials and quasi-periodic potentials belong to the class of almost periodic potentials, where in the description above, $\Omega$ is a compact abelian group, $\mu$ is the normalized Haar measure on this group, $\{ \mathcal{T}_x \}_{x \in \R}$ is a translation flow, and $f$ is continuous. Note that a bounded continuous function $q : \R \to \R$ is called almost periodic if any sequence of translates, $q_k (\cdot) = q(x - x_k)$, contains a subsequence that converges uniformly. This description is equivalent to the formulation above in the sense that any such function belongs to some almost periodic family $\{ q_\omega \}_{\omega \in \Omega}$, and for every almost periodic family $\{ q_\omega \}_{\omega \in \Omega}$, every $q_\omega$ is almost periodic.

In light of the existing examples with non-trivial absolutely continuous spectrum, the following problem became quite popular in the community.

\begin{klc}
If the family $\{ L_{q_\omega} \}_{\omega \in \Omega}$ is such that $\Sigma_\mathrm{ac} \not= \emptyset$, then the potentials $\{ q_\omega \}_{\omega \in \Omega}$ are almost periodic.
\end{klc}

This question is attributed to Kotani and Last since both of them were heavily interested in this problem and advocated it in the community. The problem has appeared in print many times, for example in surveys by Damanik~\cite{D07} and Jitomirskaya~\cite{J07}, and in Kotani-Krishna~\cite{KK88} and Simon~\cite{S07}.

Belief in the Kotani-Last conjecture was nurtured in a least two ways. On the one hand, as indicated above, at the time \cite{D07, J07, KK88, S07} were written, all known examples of ergodic families $\{ L_{q_\omega} \}_{\omega \in \Omega}$ with $\Sigma_\mathrm{ac} \not= \emptyset$ were indeed almost periodic. On the other hand, Kotani has shown (in the context of what is nowadays called Kotani theory) that $\Sigma_\mathrm{ac} \not= \emptyset$ implies that the potentials are deterministic, that is, the past (= the restriction to the left half-line) determines the future (= the restriction to the right half-line). Indeed, they are reflectionless on the set $\Sigma_\mathrm{ac}$, and the relation between the half-line Weyl-Titchmarsh functions this entails allows one to determine the future from the past; see \cite{K84, K97}.

In addition, one may consider the following instructive simple case. If the base flow is a simple translation flow on the circle, then there is indeed a transition to purely singular spectrum as soon as one gives up the continuity of the sampling function, and this is precisely where almost periodicity is lost \cite{DK05}. The paper \cite{DK05} is an extension of Kotani \cite{K89} who had looked at potentials taking finitely many values. Moreover, by approximating continuous sampling functions with discontinuous ones, one can show that the absence of absolutely continuous spectrum is generic within the space of continuous sampling functions \cite{AD05}.\footnote{To be precise, the papers mentioned in this paragraph consider discrete Schr\"odinger operators. The ideas developed in these papers have counterparts in the continuum case.}

In fact, Kotani theory exists in several other settings. For example, it was worked out for discrete Schr\"odinger operators by Simon \cite{S83}, for a class of operators including Jacobi matrices by Minami \cite{M86}, and for extended CMV matrices by Geronimo \cite{G93} and Geronimo-Teplyaev \cite{GT94}. In all of these settings the Kotani-Last conjecture makes sense and takes essentially the same form. Namely, if there is non-trivial absolutely continuous spectrum, then the coefficients must be almost periodic. The supporting evidence is also two-fold in these other settings (concrete positive examples, and a determinism property stemming from reflectionlessness), just as in the continuum Schr\"odinger context.

In two independent (and almost simultaneously released) recent papers, the Kotani-Last conjecture was disproved in two of these scenarios. Namely, Avila \cite{A14} worked with successive deformations of periodic potentials to construct non-almost periodic examples for which the absolutely continuous spectrum is non-empty. He does this for both continuum and discrete Schr\"odinger operators. On the other hand, Volberg and Yuditskii \cite{VY14} considered ergodic families of Jacobi matrices and employed an approach based on inverse spectral theory to characterize, within a certain class of Jacobi matrices, those families that are not almost periodic and do have (purely) absolutely continuous spectrum.

Thus, the Kotani-Last conjecture has been shown to fail in the settings of continuum and discrete Schr\"odinger operators as well as Jacobi matrices (of course, the failure within the class of discrete Schr\"odinger operators implies the failure within the class of Jacobi matrices). Moreover, there are two completely different ways of establishing this, one based on direct spectral analysis, and another one based on inverse spectral theory.

Since the continuum Schr\"odinger operator setting is the one Kotani originally worked in, it is an important question whether the Volberg-Yuditskii approach can be used in this setting. Our main purpose in this paper is to show that this is possible. That is, akin to Volberg and Yuditskii in \cite{VY14}, we will consider a certain class of ergodic families of continuum Schr\"odinger operators with purely absolutely continuous spectrum, and characterize those families for which the potentials are not almost periodic (and of course also show that the latter class is non-empty).

\subsection{Setting and Main Result}

Let us describe the precise setting we will work in. We will consider subsets $E$ of the real line of the form
\begin{equation}\label{e.eform}
E = [-1, \infty) \setminus \bigcup_{k = 1}^\infty (\lambda_k^-, \lambda_k^+),
\end{equation}
where the parameters $\lambda^\pm_k$ satisfy
\begin{equation}\label{e.lambdacond}
\lambda_0^+ := -1  < \lambda_1^- < \lambda_1^+ < \lambda_2^- < \lambda_2^+ < \cdots < \infty =: \lambda_0^-.
\end{equation}
We require that $\sup_{k \ge 1} \lambda_k^+ < \infty$, and hence modulo energy reparametrization we assume for the sake of simplicity and without loss of generality that
\begin{equation}\label{e.econtainshalfline}
\sup_{k \ge 1} \lambda_k^+ = 0,
\end{equation}
that is, $E$ contains the right half-line and $0$ is accumulated from the left by gaps of $E$. Let $\mathcal{D} = \bar \C \setminus E$.

We give a parametric description of our class of domains by means of special conformal mappings. Consider the so called comb domain $\Pi=\Pi (\{\omega_k, h_k\}_{k=1}^\infty)$,
\begin{equation}\label{e.Thetadefinition}
\Pi:= \{ w \in \C : \Re w > 0, \; \Im w > 0 \} \setminus \{ \omega_k + i y : 0 < y \le h_k : k \ge 1 \},
\end{equation}
with frequencies $\{ \omega_k \}_{k \ge 1} \subset (0,\infty)$, $\omega_1<\dots<\omega_k<\omega_{k+1}<\dots$, and heights $\{ h_k \}_{k \ge 1} \subset (0, \infty)$ such that
\begin{equation}\label{e.discreteomegaks}
\lim_{k \to \infty} \omega_k =: \omega_* < \infty
\end{equation}
and
\begin{equation}\label{e.widomcondition}
\sum_{k \ge 1} h_k < \infty.
\end{equation}

By the Riemann mapping theorem, $\Pi$ can be mapped conformally on the upper half-plane. The inverse  conformal mapping $\Theta : \C_+ \to \Pi$ is defined uniquely by the following normalization,
\begin{equation}\label{e.Thetanormalization}
\Theta(\lambda) \sim \sqrt{\lambda} \text{ as } \lambda \to -\infty, \quad \Theta(-1) = 0.
\end{equation}
$\Theta$ has a continuous extension to the closed half-plane,  and we define  $E=\Theta^{-1}(\R_+)$. Evidently, the conditions \eqref{e.eform} and \eqref{e.lambdacond} hold automatically. To satisfy \eqref{e.econtainshalfline} we consider a rescaling of the comb, if necessary.

Let us point out that the condition \eqref{e.widomcondition} expresses the fact that
$
\mathcal{D}
$
is a Widom domain, see below.

Our last major assumption is the independence condition on the frequencies. Let
$$
\T^\infty=\{\alpha=\{\alpha_k\}_{k=1}^\infty:\ \alpha_k\in\T\},
$$
where $\T$ denotes the unit circle, $\T = \{ w \in \C : |w|=1 \}$. We consider the translation flow
\begin{equation}\label{pyu1}
\mathcal S_{\ell} \alpha = \{ \alpha_k e^{-2i\omega_k \ell} \}_{k=1}^\infty, \ \alpha \in \T^\infty, \; \ell \in \R.
\end{equation}
We say that the frequencies are independent if the closure of the trajectory $\{\mathcal S_\ell\alpha\}_{\ell\in\R}$ coincides with the whole torus,
\begin{equation}\label{e.frequenciesindependent}
\overline{\{\mathcal S_\ell\alpha : \ell\in\R\}} = \T^\infty,
\end{equation}
for some, and therefore for all, $\alpha\in\T^\infty$. Evidently this condition is stable with respect to a rescaling of $\Pi$.

Throughout the paper we will assume that a set $E$ satisfying \eqref{e.discreteomegaks}, \eqref{e.widomcondition} and \eqref{e.frequenciesindependent} is given. We want to study continuum one-dimensional Schr\"odinger operators
$$
L_q = - \frac{d^2}{dx^2} + q
$$
in $L^2(\R)$ whose spectrum is given by the set $E$ and which are reflectionless on $E$, that is,
$$
\lim_{\varepsilon \downarrow 0} m_+(\lambda + i \varepsilon) = - \lim_{\varepsilon \downarrow 0} \overline{m_-(\lambda + i \varepsilon)} \quad \text{for almost every } \lambda \in E,
$$
where $m_\pm$ is the Weyl-Titchmarsh function of the half-line restriction of $L_q$ to $\R_\pm$ (we recall the definition of $m_\pm$ in Section~\ref{s.2}). This class of operators has been studied, for example, in \cite{C89, K97, K08, K14}.

Let us point out that in our setting every such operator $L_q$ has purely absolutely continuous spectrum; see, for example, \cite{PR09}. We set
$$
q(E) := \{ q \in C_b(\R) : \sigma(L_q) = E, \; L_q \text{ is reflectionless on } E \}.
$$

It turns out that it is beneficial also to consider the following conformal map. Choose
\begin{equation}\label{e.lambda*choice}
\lambda_* < -1
\end{equation}
and consider the change of variables
\begin{equation}\label{e.changeofvariables}
z = \frac{1}{\lambda - \lambda_*}.
\end{equation}
Throughout this paper, whenever in some formula or expression both $\lambda$ and $z$ appear, they are linked via \eqref{e.changeofvariables}. Note that \eqref{e.changeofvariables} maps the upper half-plane to the lower half-plane.

The change of variables \eqref{e.changeofvariables} transforms $E$ and $\mathcal{D}$, respectively, into
$$
\tilde E = \left\{ \frac{1}{\lambda-\lambda_* } : \lambda \in E \right\}, \quad \tilde{\mathcal{D}} = \left\{ \frac{1}{\lambda-\lambda_*} : \lambda \in \mathcal{D} \right\}.
$$
In particular, the gap $(\lambda_k^-,\lambda_k^+)$ of $E$ corresponds to the gap $(z_k^-,z_k^+)$ of $\tilde E$, where $z_k^\pm = \frac{1}{\lambda_k^\mp - \lambda_*}$. Let $\D / \Gamma \simeq \tilde{\mathcal{D}}$ be a uniformization of $\tilde{\mathcal{D}}$, that is, $\Gamma$ is a Fuchsian group and $\mathfrak{z} : \D \to \tilde{\mathcal{D}}$ is a meromorphic function with $\mathfrak{z} \circ \gamma = \mathfrak{z}$ for every $\gamma \in \Gamma$ and such that
\begin{itemize}

\item $\forall z \in \tilde{\mathcal{D}} \; \exists \zeta \in \D : \mathfrak{z}(\zeta) = z$,

\item $\mathfrak{z}(\zeta_1) = \mathfrak{z}(\zeta_2) \Rightarrow \exists \gamma \in \Gamma : \zeta_1 = \gamma(\zeta_2)$.

\end{itemize}
The Fuchsian group $\Gamma$ is equivalent to the fundamental group $\Gamma_{\tilde{\mathcal{D}}}$ of $\tilde{\mathcal{D}}$. We denote its group of characters by $\Gamma^*$. We normalize $\mathfrak{z}$ by $\mathfrak{z}(0) = \infty$ and $(\zeta\mathfrak z)(0)>0$. Composing the maps, we obtain the uniformization $\D / \Gamma \simeq \mathcal{D}$ via $\lambda_* + \frac{1}{\mathfrak{z}} : \D \to \mathcal{D}$.

Recall that a meromorphic function $f$ in the disk $\D$ is said to be of bounded characteristic if it can be represented as the ratio of two bounded analytic functions, $f = f_1/f_2$. The function $f$ is of Smirnov class if in addition $f_2$ is an outer function. Every meromorphic function $F$ in $\tilde{\mathcal{D}}$ defines an automorphic function $f = F \circ \mathfrak{z}$ (i.e., $f \circ \gamma = f$ for every $\gamma \in \Gamma$) and vice versa. We say that $F$ belongs to the Smirnov class $N_+(\tilde{\mathcal{D}})$ if $f = F \circ \mathfrak{z}$ is of Smirnov class in $\D$. Note that Lebesgue measure on $\T$ corresponds to the harmonic measure on $\tilde E$. Thus, every $F \in N_+(\tilde{\mathcal{D}})$ has boundary values for Lebesgue almost all $x \in \tilde E$.

\begin{defi}\label{d.dct}
We say that $E$ satisfies DCT if for every $F \in N_+(\tilde{\mathcal{D}})$ with
$$
\oint_{\partial \tilde{\mathcal{D}}} |F(x)| \, |dx| < \infty \quad \text{ and } F(\infty) = 0,
$$
we have
$$
\frac{1}{2 \pi i} \oint_{\partial \tilde{\mathcal{D}}} \frac{F(x)}{x - z} \, dx = F(z)
$$
for every $z \in \tilde{\mathcal{D}}$.
\end{defi}

We refer the reader to Hasumi \cite{H83} for an extensive discussion of this condition on the domain. It is important to note that the condition DCT does not hold for all Widom domains.

\bigskip

We are now able to state our main theorem.

\begin{theorem}\label{t.main}
Consider a comb domain $\Pi(\{\omega_k,h_k\}_{k=1}^\infty)$ subject to the assumptions \eqref{e.discreteomegaks}, \eqref{e.widomcondition}, and \eqref{e.frequenciesindependent}. Suppose $E$ is the image of $\R_+$ under {\rm (}the continuous extension to the closure of{\rm )} the conformal map $\Theta^{-1}$ sending the comb domain to the upper half-plane. Then we have the following dichotomy:
\begin{itemize}

\item[{\rm (a)}] If $E$ satisfies DCT, then every $q \in q(E)$ is almost periodic.

\item[{\rm (b)}] If $E$ does not satisfy DCT, then every $q \in q(E)$ is not almost periodic.

\end{itemize}
\end{theorem}

Of course our main interest here is in sets $E$ for which the second alternative applies. We will show that such sets indeed exist and hence establish the following direct consequence of the main theorem.

\begin{coro}\label{c.main}
There exist continuous bounded ergodic potentials, which are not almost periodic, for which the associated Schr\"odinger operator has purely absolutely continuous spectrum.
\end{coro}

This provides counterexamples to the Kotani-Last conjecture for continuum Schr\"odinger operators. We emphasize again that Artur Avila has already constructed counterexamples to this conjecture \cite{A14}, but our main point here is to show how one may obtain them using the framework suggested and developed by Volberg and Yuditskii in \cite{VY14}.

\subsection{Structure of the Paper}

In Section~\ref{s.2} we establish a connection between continuum Schr\"odinger operators that are reflectionless on $E$ and Jacobi matrices that are reflectionless on some derived set $\tilde E$. In Section~\ref{s.3} we summarize the main aspects of the work of Volberg and Yuditskii \cite{VY14} that are relevant to our work and then establish a correspondence between potentials $q$ in $q(E)$ and character-automorphic Hardy spaces $H^2(\alpha)$, by going through the Jacobi matrix that corresponds to each of them. This gives rise to the map $\pi: q \mapsto \alpha$. We also establish a unitary equivalence of the Schr\"odinger operator with potential $q$ in $L^2(\R)$ with a multiplication operator in a character-automorphic $L^2$-space with the property that $L^2(\R_+)$ corresponds to the character-automorphic Hardy space $H^2(\alpha)$; see Theorem~\ref{t,spectral}. These unitary maps give rise to a chain of subspaces which is a natural generalization of the classical Payley-Wiener chain of subspaces in Fourier analysis. In Section~\ref{s.4} we study the translation flow on the space of potentials $q$, and the induced flow on the space of characters $\alpha$. It turns out that the induced flow is an explicit translation flow on an infinite-dimensional torus. It therefore becomes essential to study the fibers $\pi^{-1}(\{ \alpha \})$. In Section~\ref{s.5} we show that almost periodicity fails precisely when there exist non-trivial fibers, that is, when $\pi$ is not 1-1. The latter condition is equivalent to the failure of DCT, and hence Theorem~\ref{t.main} follows. We also show that there exist sets $E$ for which there are non-trivial fibers, thus proving Corollary~\ref{c.main}. Finally, we work out the analog of these results for extended CMV matrices in the appendix. In particular, we formulate the Kotani-Last conjecture for extended CMV matrices and disprove it there, following the same overall strategy.

\section{Connecting Continuum Schr\"odinger Operators with Jacobi Matrices}\label{s.2}

Consider a continuum Schr\"odinger operator on the line,
$$
Ly = -y'' + qy,
$$
with $q \in C_b(\R)$. For $\lambda \in \C$, consider the associated differential equation
\begin{equation}\label{e.eve}
-u''(x) + q(x) u(x) = \lambda u(x).
\end{equation}
The solution space of \eqref{e.eve} is two-dimensional, and a standard basis of this space is given by $\{ u_1(\cdot,\lambda), u_2(\cdot,\lambda)\}$, where $u_1(\cdot,\lambda)$ and $u_2(\cdot,\lambda)\}$ solve \eqref{e.eve} and obey the initial conditions
$$
\begin{pmatrix} u_1(0,\lambda) & u_2(0,\lambda) \\ u'_1(0,\lambda) & u_2'(0,\lambda) \end{pmatrix} = \begin{pmatrix} 1 & 0 \\ 0 & 1 \end{pmatrix}.
$$
If $\lambda \not\in \sigma(L)$, then there are solutions $u_\pm(\cdot, \lambda)$ of \eqref{e.eve} so that $u_\pm$ is square-integrable near $\pm \infty$. In fact, using the constancy of the Wronskian one sees that these solutions are unique up to a multiplicative constant. Moreover, if $\lambda \not\in \R$, then $u_\pm$ cannot be a multiple of $u_2$ (for otherwise a self-adjoint half-line problem would have a non-real eigenvalue). This implies that the multiplicative constant for $u_\pm$ can be chosen in such a way that for suitable $m_\pm(\lambda) \in \C$, we have
$$
u_\pm(\cdot,\lambda) = u_1(\cdot,\lambda) \pm m_\pm(\lambda) u_2(\cdot,\lambda).
$$
With this normalization, the Wronskian of $u_+$ and $u_-$ turns out to be
$$
W(u_+(\cdot,\lambda), u_-(\cdot,\lambda)) = - (m_+(\lambda) + m_-(\lambda)).
$$

\begin{lemma}\label{l.mfunctionenergyvariation}
We have
\begin{align}
\label{e.mfunctionenergyvariation1}  \int_0^\infty u_+(x,\lambda_1) u_+(x,\lambda_2) \, dx & = \frac{m_+(\lambda_1) - m_+(\lambda_2)}{ \lambda_1 - \lambda_2 }, \\
\label{e.mfunctionenergyvariation2}  \int_{-\infty}^0 u_-(x,\lambda_1) u_-(x,\lambda_2) \, dx & =
 \frac{m_-(\lambda_1) - m_-(\lambda_2)}{ \lambda_1 - \lambda_2 }.
\end{align}
\end{lemma}

\begin{proof}
Notice that
\begin{align*}
\big( u_+(x,\lambda_1) & u_+'(x,\lambda_2) - u_+'(x,\lambda_1) u_+(x,\lambda_2) \big)' \\
& = u_+(x,\lambda_1) u_+''(x,\lambda_2) - u_+''(x,\lambda_1) u_+(x,\lambda_2) \\
& = u_+(x,\lambda_1) (q(x) - \lambda_2) u_+(x,\lambda_2) - (q(x) - \lambda_1) u_+(x,\lambda_1) u_+(x,\lambda_2) \\
& = ( \lambda_1 - \lambda_2 ) u_+(x,\lambda_1) u_+(x,\lambda_2).
\end{align*}
Integrating from $0$ to $\infty$, we find
$$
( \lambda_1 - \lambda_2 ) \int_0^\infty u_+(x,\lambda_1) u_+(x,\lambda_2) \, dx = - u_+(0,\lambda_1) u_+'(0,\lambda_2) + u_+'(0,\lambda_1) u_+(0,\lambda_2),
$$
which implies \eqref{e.mfunctionenergyvariation1} due to the initial values of $u_+$. The proof of \eqref{e.mfunctionenergyvariation2} is analogous.
\end{proof}

Moreover, the resolvent $(L-\lambda)^{-1}$ is an integral operator with kernel
\begin{equation}\label{e.greensfunction}
G(x,y;\lambda) = - \frac{u_+(\max\{x,y\},\lambda) u_-(\min\{x,y\},\lambda)}{m_+(\lambda) + m_-(\lambda)}.
\end{equation}

Since $q$ is bounded,  the whole-line and the half-line spectra are bounded from below. We fix a real $\lambda_*$ that lies below all these spectra. Then we still have solutions $u_\pm(\cdot,\lambda_*)$, which are now real-valued, with the desired square-integrability properties as above. The expressions above continue to hold in the way stated, except \eqref{e.mfunctionenergyvariation1}--\eqref{e.mfunctionenergyvariation2}, which become
\begin{equation}\label{e.normofupreal}
\int_0^\infty u_+(x,\lambda_*)^2 \, dx = m_+'(\lambda_*)
\end{equation}
and
\begin{equation}\label{e.normofumreal}
\int_{-\infty}^0 u_-(x,\lambda_*)^2 \, dx = m_-'(\lambda_*).
\end{equation}

Next consider the decomposition $L^2(\R) = L^2(\R_-) \oplus L^2(\R_+)$ and the induced decomposition of $A := (L - \lambda_*)^{-1}$. Let us compute the off-diagonal terms.

\begin{lemma}
Set
$$
e_\pm = \frac{\chi_{\R_\pm} u_\pm(\cdot,\lambda_*)}{\|\chi_{\R_\pm} u_\pm\|}
$$
and
\begin{equation}\label{e.a0def}
a_0 = - \frac{\left( m_+'(\lambda_*) m_-'(\lambda_*) \right)^{1/2}}{m_+(\lambda_*) + m_-(\lambda_*)}.
\end{equation}
Then, we have
\begin{equation}\label{e.Adecomposition}
A = \begin{pmatrix} A_{-} & 0 \\ 0 & A_{+} \end{pmatrix} + a_0 \langle\,  \cdot \, , e_+ \rangle \, e_-(x) + a_0
\langle \,\cdot\, ,e_- \rangle \, e_+(x).
\end{equation}
\end{lemma}

\begin{proof}
Note first that due to \eqref{e.normofupreal}, we have
$$
\|\chi_{\R_+} u_+(\cdot,\lambda_*)\| = \left( m_+'(\lambda_*) \right)^{1/2}.
$$

Suppose $\varphi_+ \in L^2(\R_+)$. Then, for $x < 0$, we have
\begin{align*}
[A \varphi_+] (x) & = \int_0^\infty G(x,y;\lambda_*) \varphi_+(y) \, dy \\
& = \int_0^\infty - \frac{u_+(y,\lambda_*) u_-(x,\lambda_*)}{m_+(\lambda_*) + m_-(\lambda_*)} \varphi_+(y) \, dy \\
& = - \frac{\langle \varphi_+, u_+(\cdot,\lambda_*) \rangle}{m_+(\lambda_*) + m_-(\lambda_*)} u_-(x,\lambda_*) \\
& = - \frac{\left( m_+'(\lambda_*) m_-'(\lambda_*) \right)^{1/2}}{m_+(\lambda) + m_-(\lambda)}
\langle \varphi_+,  e_+  \rangle \, e_-(x)
\end{align*}
due to \eqref{e.greensfunction}. The computation of the other off-diagonal term is analogous.
\end{proof}

\begin{lemma}
Let $\lambda \in \rho (L)$ and let $z$ be given by \eqref{e.changeofvariables}. Then, the matrix
\begin{equation}\label{e.Rzmatrixdef}
R(z) = \begin{pmatrix} \langle (A-z)^{-1} e_- , e_- \rangle & \langle (A-z)^{-1} e_+ , e_- \rangle \\ \langle (A-z)^{-1} e_- , e_+ \rangle & \langle (A-z)^{-1} e_+ , e_+ \rangle \end{pmatrix}
\end{equation}
has the form
\begin{equation}\label{e.aminuszinversematrix}
R(z)^{-1} = \begin{pmatrix} r_-(z)^{-1} & a_0 \\ a_0 & r_+(z)^{-1} \end{pmatrix}^{-1},
\end{equation}
where $a_0$ is as in \eqref{e.a0def} and
\begin{align}
\label{e.rminusdef} r_-(z) & = \frac{m_+(\lambda_*) + m_-(\lambda_*)}{m_-'(\lambda_*)} \cdot \frac{m_-(\lambda_*) - m_-(\lambda)}{m_+(\lambda_*) + m_-(\lambda)}, \\
\label{e.rplusdef} r_+(z) & = \frac{m_+(\lambda_*) + m_-(\lambda_*)}{m_+'(\lambda_*)} \cdot \frac{m_+(\lambda_*) - m_+(\lambda)}{m_-(\lambda_*) + m_+(\lambda)}.
\end{align}
\end{lemma}

\begin{proof}
Recall that $A = (L - \lambda_*)^{-1}$. Thus, by the resolvent identity,
$$
(A-z)^{-1} = (L - \lambda_*) (\lambda_* - \lambda) (L - \lambda)^{-1} = (L - \lambda + \lambda - \lambda_*) (\lambda_* - \lambda) (L - \lambda)^{-1},
$$
and hence
\begin{equation}\label{e.aminuszlminuslambda}
(A-z)^{-1} = (\lambda_* - \lambda) - (\lambda_* - \lambda)^2 (L - \lambda)^{-1}.
\end{equation}

Thus, in order to determine the matrix elements in \eqref{e.Rzmatrixdef} it suffices to consider $(L - \lambda)^{-1} \chi_{\R_\pm} u_\pm$. Denote
\begin{equation}\label{e.f3termexpansion}
f = (L - \lambda)^{-1} \chi_{\R_\pm} u_\pm.
\end{equation}
Since ($f \in D(L)$ and) $(L-\lambda)f = \chi_{R_\pm} u_\pm$, there must be constants $c_1^\pm, c_2^\pm$ such that
\begin{equation}\label{e.f3termexpansion2}
f = c_1^\pm \chi_{\R_-} u_-(x,\lambda) + \frac{1}{\lambda_* - \lambda} \chi_{\R_\pm} u_\pm(x,\lambda_*) + c_2^\pm \chi_{\R_+} u_+(x,\lambda).
\end{equation}
Indeed, for each expression like the one on the right-hand side, $(L-\lambda)f = \chi_{\R_\pm} u_\pm$ holds pointwise, and in addition there are unique choices of $c_1^\pm,c_2^\pm$ such that the expression determines a function in $D(L)$.

If we combine \eqref{e.aminuszlminuslambda} and \eqref{e.f3termexpansion}, we find
\begin{align}
\label{e.aminusz1} (A-z)^{-1} \chi_{\R_+} u_+ & = - (\lambda_* - \lambda)^2 \begin{pmatrix} \chi_{\R_-} u_-(\cdot,\lambda) & \chi_{\R_+} u_+(\cdot,\lambda) \end{pmatrix} \begin{pmatrix} c_1^+ \\ c_2^+ \end{pmatrix}, \\
\label{e.aminusz2} (A-z)^{-1} \chi_{\R_-} u_- & = - (\lambda_* - \lambda)^2 \begin{pmatrix} \chi_{\R_-} u_-(\cdot,\lambda) & \chi_{\R_+} u_+(\cdot,\lambda) \end{pmatrix} \begin{pmatrix} c_1^- \\ c_2^- \end{pmatrix}.
\end{align}

Let us determine these unique values of $c_1^\pm,c_2^\pm$. The only issue is to ensure the continuity of the function and of the derivative at zero. Continuity at zero leads to the requirement
$$
c_1^\pm = \pm \frac{1}{\lambda_* - \lambda} + c_2^\pm.
$$
Continuity of the derivative at zero leads to the requirement
$$
- c_1^\pm m_-(\lambda) = \frac{1}{\lambda_* - \lambda} m_\pm(\lambda_*) + c_2^\pm m_+(\lambda).
$$
Thus, we seek $c_1^\pm,c_2^\pm$ with
\begin{equation}\label{e.matrixforc1c2}
\begin{pmatrix} 1 & -1 \\ m_-(\lambda) & m_+(\lambda) \end{pmatrix}  \begin{pmatrix} c_1^\pm \\ c_2^\pm \end{pmatrix} = \frac{1}{\lambda_* - \lambda} \begin{pmatrix} \pm 1 \\ - m_\pm(\lambda_*)\end{pmatrix}.
\end{equation}
Since $m_\pm$ are Herglotz functions, the determinant $m_+(\lambda) + m_-(\lambda)$ of the matrix $M(\lambda)$ appearing on the left-hand side of \eqref{e.matrixforc1c2} is non-zero for every $\lambda \in \C_+$. This shows that $c_1^\pm,c_2^\pm$ are determined uniquely by
\begin{equation}\label{e.cpm12}
\begin{pmatrix} c_1^- & c_1^+ \\ c_2^- & c_2^+ \end{pmatrix} = - \frac{1}{\lambda_* - \lambda} M(\lambda)^{-1} M(\lambda_*).
\end{equation}

Combining \eqref{e.aminusz1} and \eqref{e.aminusz2} with \eqref{e.cpm12} and using Lemma~\ref{l.mfunctionenergyvariation}, we obtain
\begin{align*}
& \begin{pmatrix} \langle (A-z)^{-1} \chi_{\R_-} u_- , \chi_{\R_-} u_- \rangle & \langle (A-z)^{-1} \chi_{\R_+} u_+ , \chi_{\R_-} u_- \rangle \\ \langle  (A-z)^{-1} \chi_{\R_-} u_- , \chi_{\R_+} u_+ \rangle & \langle (A-z)^{-1} \chi_{\R_+} u_+ , \chi_{\R_+} u_+ \rangle \end{pmatrix} \\
& = \begin{pmatrix} \langle \chi_{\R_-} u_- | \\ \langle \chi_{\R_+} u_+ | \end{pmatrix} \begin{pmatrix} | (A-z)^{-1} \chi_{\R_-} u_- \rangle &  | (A-z)^{-1} \chi_{\R_+} u_+ \rangle \end{pmatrix} \\
& = - (\lambda_* - \lambda)^2 \begin{pmatrix} \langle \chi_{\R_-} u_- | \\ \langle \chi_{\R_+} u_+ | \end{pmatrix} \begin{pmatrix} | \chi_{\R_-} u_-(\cdot,\lambda) \rangle & | \chi_{\R_+} u_+(\cdot,\lambda) \rangle \end{pmatrix} \begin{pmatrix} c_1^- & c_1^+ \\ c_2^- & c_2^+ \end{pmatrix} \\
& = \begin{pmatrix}  m_-(\lambda_*) - m_-(\lambda) & 0 \\ 0 &  m_+(\lambda_*) - m_+(\lambda) \end{pmatrix} M(\lambda)^{-1} M(\lambda_*).
\end{align*}
Here we use the notation $\langle a | b \rangle := \langle b, a \rangle$. Therefore,
\begin{align*}
R(z) & = \begin{pmatrix}  \frac{\langle (A-z)^{-1} \chi_{\R_-} u_- , \chi_{\R_-} u_- \rangle}{\|\chi_{\R_-} u_-\|^{2}} &  \frac{\langle (A-z)^{-1} \chi_{\R_+} u_+ , \chi_{\R_-} u_- \rangle}{\|\chi_{\R_-} u_-\| \|\chi_{\R_+} u_+\|} \\  \frac{\langle  (A-z)^{-1} \chi_{\R_-} u_- , \chi_{\R_+} u_+ \rangle}{\|\chi_{\R_-} u_-\| \|\chi_{\R_+} u_+\|} & \frac{\langle (A-z)^{-1} \chi_{\R_+} u_+ , \chi_{\R_+} u_+ \rangle}{\|\chi_{\R_+} u_+\|^2} \end{pmatrix} \\
& = \begin{pmatrix}  \frac{m_-(\lambda_*) - m_-(\lambda)}{(m_-'(\lambda_*))^{1/2}} & 0 \\ 0 &  \frac{m_+(\lambda_*) - m_+(\lambda)}{(m_+'(\lambda_*))^{1/2}} \end{pmatrix} M(\lambda)^{-1} M(\lambda_*) \begin{pmatrix} (m_-'(\lambda_*))^{-1/2} & 0 \\ 0 & (m_+'(\lambda_*))^{-1/2} \end{pmatrix},
\end{align*}
that is,
\begin{align*}
R(z)^{-1} & = \begin{pmatrix} (m_-'(\lambda_*))^{1/2} & 0 \\ 0 & (m_+'(\lambda_*))^{1/2} \end{pmatrix} M(\lambda_*)^{-1} M(\lambda) \begin{pmatrix}  \frac{(m_-'(\lambda_*))^{1/2}}{m_-(\lambda_*) - m_-(\lambda)} & 0 \\ 0 &  \frac{(m_+'(\lambda_*))^{1/2}}{m_+(\lambda_*) - m_+(\lambda)} \end{pmatrix} \\
& = \frac{1}{m_+(\lambda_*) + m_-(\lambda_*)} \begin{pmatrix} (m_-'(\lambda_*)) \frac{m_+(\lambda_*) + m_-(\lambda)}{m_-(\lambda_*) - m_-(\lambda)} & - (m_+'(\lambda_*) m_-'(\lambda_*))^{1/2} \\ - (m_+'(\lambda_*) m_-'(\lambda_*))^{1/2} & (m_+'(\lambda_*)) \frac{m_-(\lambda_*) + m_+(\lambda)}{m_+(\lambda_*) - m_+(\lambda)} \end{pmatrix} \\
& = \begin{pmatrix}  \frac{(m_-'(\lambda_*))(m_+(\lambda_*) + m_-(\lambda))}{(m_+(\lambda_*) + m_-(\lambda_*))(m_-(\lambda_*) - m_-(\lambda))} & a_0 \\ a_0 & \frac{(m_+'(\lambda_*)) (m_-(\lambda_*) + m_+(\lambda))}{(m_+(\lambda_*) + m_-(\lambda_*)) (m_+(\lambda_*) - m_+(\lambda))} \end{pmatrix},
\end{align*}
which proves \eqref{e.aminuszinversematrix} and concludes the proof.
\end{proof}

\begin{remark}\label{r.mplusminusrplusminuscorrespondence}
From \eqref{e.a0def} and \eqref{e.rminusdef}--\eqref{e.rplusdef}, we find
\begin{align}
\label{e.mplustorplus} a_0^2 r_+(z) & = \frac{m_-'(\lambda_*)}{m_+(\lambda_*) + m_- (\lambda_*)} \cdot \frac{m_+(\lambda_*) - m_+(\lambda)}{m_-(\lambda_*) + m_+(\lambda)}, \\
\label{e.mminustorminus} - \frac{1}{r_-(z)} & = - \frac{m_-'(\lambda_*)}{m_+(\lambda_*) + m_- (\lambda_*)} \cdot \frac{m_+(\lambda_*) + m_-(\lambda)}{m_-(\lambda_*) - m_-(\lambda)}.
\end{align}
\end{remark}

\begin{lemma}\label{l.cyclic}
The vectors $\{ e_-, e_+ \}$ are cyclic for $A$. That is, if $f \in L^2(\R)$ is such that
$$
\langle f, (A - \bar z)^{-1} e_- \rangle = \langle f, (A - \bar z)^{-1} e_+ \rangle = 0
$$
for every $z \in \C \setminus \R$, then $f = 0$.
\end{lemma}

\begin{proof}
The assumption is equivalent to
$$
\langle f, (A - \bar z)^{-1} (\chi_{\R_-} u_-(\cdot,\bar \lambda)) \rangle = \langle f, (A - \bar z)^{-1} (\chi_{\R_+} u_+(\cdot,\bar \lambda)) \rangle = 0
$$
for every $z \in \C \setminus \R$. Due to \eqref{e.aminuszlminuslambda}--\eqref{e.f3termexpansion2} and the invertibility of the matrix \eqref{e.cpm12}, we may deduce that
$$
\langle f, \chi_{\R_-} u_-(\cdot,\bar \lambda) \rangle = \langle f, \chi_{\R_+} u_+(\cdot,\bar \lambda) \rangle = 0
$$
for every $\lambda \in \C \setminus \R$. This in turn implies $f = 0$; see, for example, \cite[Chapter~2]{L87}.
\end{proof}

\begin{remark}\label{r.fromqtoJ}
Lemma~\ref{l.cyclic} (together with the spectral theorem) shows that there is $2 \times 2$-matrix measure $d\sigma$ on $\R$ such that
$$
R(z) = \int \frac{d\sigma(x)}{x-z},
$$
where $R(z)$ is the matrix defined in \eqref{e.Rzmatrixdef}. Moreover, there exists a unitary map
$$
\mathcal{F}_A : L^2(\R) \to L^2(d\sigma)
$$
such that
$$
\mathcal{F}_A e_- = \begin{pmatrix} 1 \\ 0 \end{pmatrix}, \quad \mathcal{F}_A e_+ = \begin{pmatrix} 0 \\ 1 \end{pmatrix}, \quad (\mathcal{F}_A (Af))(x) = x (\mathcal{F}_A f) (x), \; f \in L^2(\R).
$$
Since $A$ is a bounded self-adjoint operator (and $d\sigma$ is compactly supported), we can consider the canonical Jacobi matrix model for the operator in $L^2(d\sigma)$ given by the multiplication with the independent variable. In other words, there are bounded sequences $\{ a_n \}_{n \in \Z} \subset (0,\infty)$, $\{ b_n \}_{n \in \Z} \subset \R$ such that with the associated Jacobi matrix $J = J(\{a_n\},\{b_n\})$, acting in $\ell^2(\Z)$ as a bounded self-adjoint operator, we have for a suitable unitary map
$$
\mathcal{F}_J : \ell^2(\Z) \to L^2(d\sigma)
$$
the analogous properties
$$
\mathcal{F}_J e_{-1} = \begin{pmatrix} 1 \\ 0 \end{pmatrix}, \quad \mathcal{F}_J e_0 = \begin{pmatrix} 0 \\ 1 \end{pmatrix}, \quad (\mathcal{F}_J (Jf))(x) = x (\mathcal{F}_J f) (x), \; f \in \ell^2(\Z).
$$
Relative to the decomposition $\ell^2(\Z) = \ell^2(\Z_-) \oplus \ell^2(\Z_+)$ (with $\Z_- = \{ \ldots, -3, -2, -1 \}$ and $\Z_+ = \{ 0, 1, 2, \ldots \}$), this two-sided Jacobi matrix $J$ has the form
$$
J = \begin{pmatrix} &&&&&&& \\ &&&&&&& \\ && J_- &&&&& \\ &&&&&&&  \\ &&&&& a_0 &&& \\ &&&& a_0 &&&& \\ &&&&&&& \\ &&&&&&& \\ &&&&&& J_+ & \\ &&&&&&& \\ &&&&&&& \end{pmatrix}
$$
with suitable one-sided Jacobi matrices $J_\pm$ (acting in $\ell^2(\Z_\pm)$). The connecting terms $a_0$ in this decomposition are precisely given by \eqref{e.a0def} due to \eqref{e.aminuszinversematrix}. Moreover, \eqref{e.Adecomposition} and \eqref{e.aminuszinversematrix} also imply that $r_\pm$ are the standard Weyl-Titchmarsh functions associated with $J_\pm$,
\begin{align}
\label{e.rminusandJ} r_-(z) & = \langle (J_- - z)^{-1} \delta_{-1}, \delta_{-1} \rangle_{\ell^2(\Z_-)}, \\
\label{e.rplusandJ} r_+(z) & = \langle (J_+ - z)^{-1} \delta_0, \delta_0 \rangle_{\ell^2(\Z_+)}.
\end{align}
In fact, $J_\pm$ are defined by these identities.
\end{remark}

\section{The Description of $q(E)$ in Terms of Character-Automorphic Hardy Spaces}\label{s.3}

We begin this section by recalling some definitions and results from \cite{VY14}, mainly concerning character-automorphic Hardy spaces $H^2(\alpha)$ and the correspondence between these spaces and reflectionless Jacobi matrices $J \in J(\tilde E)$. Given the results from the previous section connecting $J(\tilde E)$ and $q(E)$, this will then allow us to identify $q(E)$ with $\{ H^2(\alpha) : \alpha \in \Gamma^* \}$. One of the main objects of interest in our subsequent discussion is the generalized Abel map $\pi : q(E) \to \Gamma^*$, which is defined as the composition map
$$
q \mapsto H^2(\alpha) \mapsto \alpha.
$$

\subsection{Fundamental Definitions and Results from \cite{VY14}}

This subsection summarizes some of the definitions and results discussed in more detail in \cite{VY14}. We refer the reader to that paper, and also to  \cite{H83}, for background and further information. Recall from Section~\ref{s.1} how we associate with the set $E$ domains $\mathcal{D}$ and $\tilde{\mathcal{D}}$, and uniformizations thereof involving a Fuchsian group $\Gamma$. Recall also that $\Gamma^*$ denotes the group of characters of $\Gamma$. For $\alpha \in \Gamma^*$, we let
$$
H^\infty(\alpha) = \{ f \in H^\infty : f \circ \gamma = \alpha(\gamma) f \},
$$
where $H^\infty$ is the standard Hardy space in $\D$. The Widom condition \eqref{e.widomcondition} ensures that $H^\infty(\alpha)$ is non-trivial (i.e., it contains non-zero elements) for every $\alpha \in \Gamma^*$.

We start with some special character-automorphic functions. Fix $z_0 \in \tilde{\mathcal{D}}$ and consider the associated $\Gamma$-orbit $\mathfrak{z}^{-1}(z_0)$, that is, $\mathrm{orb}(\zeta_0) = \{ \gamma(\zeta_0) : \gamma \in \Gamma \}$ for any $\zeta_0 \in \mathfrak{z}^{-1}(z_0)$. Consider the Blaschke product $b_{z_0}$ with zeros precisely given by $\mathfrak{z}^{-1}(z_0)$ and normalized so that $b_{z_0}(0) > 0$. This is the so called Green function of the group $\Gamma$ (cf.~Pommerenke \cite{Pom}), which is related to the classical Green function $G(z,z_0)$ of the domain $\tilde{\mathcal{D}}$ by the following identity,
$$
\log \frac{1}{|b_{z_0}(\zeta)|} = G(\mathfrak z(\zeta),z_0).
$$
That is, starting from $G(z,z_0)$, we can define its harmonic conjugate function $*G(z,z_0)$. However, this function is multivalued in $\tilde{\mathcal{D}}$. In other words, $b_{z_0} = e^{-G(\mathfrak z,z_0) - i*G(\mathfrak z,z_0)}$ is character automorphic. If $\gamma_{\tilde{\mathcal{D}}}$ is an element of the fundamental group $\Gamma_{\tilde{\mathcal{D}}}$ of the domain $\tilde{\mathcal{D}}$, which is equivalent to $\Gamma$, then we have
$$
e^{-G(\gamma_{\tilde{\mathcal{D}}}(z),z_0) - i*G(\gamma_{\tilde{\mathcal{D}}}(z), z_0)} = \mu_{z_0} (\gamma_{\tilde{\mathcal{D}}}) e^{-G(z,z_0) - i*G(z,z_0)}, \ \mu_{z_0}(\gamma_{\tilde{\mathcal{D}}}) \in \T,\ \gamma_{\tilde{\mathcal{D}}} \in \Gamma_{\tilde{\mathcal{D}}},
$$
or, equivalently,
$$
b_{z_0} \circ \gamma = \mu_{z_0}(\gamma) b_{z_0}, \ \gamma \in \Gamma.
$$
This system of multipliers forms a character $\mu_{z_0}$ of the group $\Gamma$.

Let $\tilde E_k=\tilde E\cap[z_0^+,z_k^-]$. The harmonic measure $\omega(z,\tilde E_k,\tilde{\mathcal{D}})$ of the set $\tilde E_k$ in $\tilde{\mathcal{D}}$ evaluated in $z\in \tilde{\mathcal{D}}$ is the harmonic function in $\tilde{\mathcal{D}}$, which satisfies the following boundary conditions,
$$
\omega(z,\tilde E_k,\tilde{\mathcal{D}}) = \begin{cases} 1, &z\in \tilde E_k, \\ 0,& z\in \tilde E\setminus\tilde E_k. \end{cases}
$$
If $\gamma_k$ is a generator of $\Gamma$ and $(\gamma_{\tilde{\mathcal{D}}})_k$ is the corresponding contour in $\tilde{\mathcal{D}}$ around $\tilde E_k$, then
$$
\mu_{z_0}(\gamma_k)=e^{2\pi i \omega(z_0,\tilde E_k,\tilde{\mathcal{D}})}.
$$

To simplify notations we
put
\begin{equation}\label{e.bdefinition}
b(\zeta) := b_\infty (\zeta), \quad \mu = \mu_{\infty} \in \Gamma^*.
\end{equation}

The Martin function $M(z,z_0)$ is a counterpart of the Green function, but related to boundary points of the domain, $z_0 \in \partial \tilde{\mathcal{D}}$. The \textit{symmetric} Martin function can be obtained as the following limit,
$$
M(z,z_0) = \lim_{\varepsilon \to 0} \frac{G(z, z_0 + i \varepsilon) + G(z, z_0 - i \varepsilon)}{G(\infty, z_0 + i \varepsilon) + G(\infty, z_0- i \varepsilon)}.
$$
This Martin function is normalized by the condition $M(\infty,z_0) = 1$. Similarly,
$$
s_{z_0,\ell} = e^{-\ell(M(\mathfrak z,z_0) + i*M(\mathfrak z,z_0))}.
$$
defines a character automorphic inner function. We denote its character by $\chi_{z_0,\ell} \in \Gamma^*$,
$s_{z_0,\ell} \circ \gamma = \chi_{z_0,\ell}(\gamma) s_{z_0,\ell}$.

In this work the main role is played by the Martin function in $\mathcal{D}$ related to the boundary point $\infty \in \partial \mathcal{D}$, which was defined in \eqref{e.Thetanormalization}. Note that $M(\lambda) = M_\infty(\lambda) := \Im \Theta(\lambda)$ is well defined in the upper half plane and possesses a single-valued  extension in $\mathcal{D}$ by the symmetry principle. Indeed, as a result of this extension we obtain a positive harmonic function, such that $M(\lambda) = 0$ for all boundary points $\lambda \in \partial \mathcal{D}$ except infinity, where $M(\lambda) \to \infty$ as $\lambda \to -\infty$.

Let us choose a generator $(\gamma_\mathcal{D})_j$ of the group $\Gamma_\mathcal{D}$, which is the closed curve that starts at $\lambda_*$ and goes in the upper half plane to the spectral gap $(\lambda^-_j,\lambda_j^+)$ and then goes back on the symmetric path in the lower half plane. Note that the symmetry conditions in the $\Theta$-plane on the image of these gaps are of the form
\begin{eqnarray*}
\Re \Theta(\lambda) = \omega_j \ &\text{or}&\ \Theta(\lambda) + \overline{\Theta(\lambda)} = 2 \omega_j, \quad \lambda \in (\lambda^-_j,\lambda_j^+),\\
\Re \Theta(\lambda) = 0 \ &\text{or}&\ \Theta(\lambda) + \overline{\Theta(\lambda)} = 0,\quad \lambda \in (-\infty,-1).
\end{eqnarray*}
Therefore, by the symmetry principle, as a result of these two reflections we get
$$
\Theta((\gamma_\mathcal{D})_j (\lambda_*)) = \Theta(\lambda_*) + 2 \omega_j.
$$
In other words,
\begin{equation}\label{e.chielldef}
e^{i \ell \Theta \circ (\lambda_* + 1/\mathfrak z)} \circ \gamma_j = \chi_{\ell} (\gamma_j) e^{i \ell \Theta \circ (\lambda_* + 1/\mathfrak z)},\ \text{where}\ \chi_{\ell} (\gamma_j) = e^{2 i \omega_j \ell}.
\end{equation}

Generally, for the given system of generators $\gamma_j \simeq (\gamma_\mathcal{D})_j \in \Gamma_\mathcal{D}$, we set an isomorphism of $\Gamma^*$ and $\T^\infty$ by
$$
\{ \alpha_j \} = \{ \alpha(\gamma_j) \} \in \T.
$$
In particular, the shift  operation \eqref{pyu1} can be rewritten as the translation
\begin{equation}\label{pyu2}
\mathcal S_\ell \alpha = \chi^{-1}_\ell \alpha
\end{equation}
in $\Gamma^*$.

In our case the condition \eqref{e.widomcondition} is equivalent to the statement that for the critical points  $c_k \in (z_k^-,z_k^+)$ of the Green function $G(z,\infty)$, the sum of the critical values is finite,
\begin{equation}\label{pyu3}
\sum_{c_j : \nabla G(c_j,\infty) = 0} G(c_j,\infty) < \infty.
\end{equation}
And this is the classical form of the Widom condition. In the other words, the zeros of $b'$, $\{ \mathfrak{z}^{-1}(c_k) : k \ge 1 \}$, also satisfy the Blaschke condition in $\D$.

We  define the last required inner character-automorphic function by
\begin{equation}\label{e.Deltadefinition}
\Delta(\zeta) := \prod_{k \ge 1} b_{c_k}(\zeta),
\end{equation}
and set $\Delta \circ \gamma = \nu(\gamma) \Delta$ for every $\gamma \in \Gamma$. Note that due to the Widom condition, $b'$ is of Smirnov class and its inner part is given by $\Delta$.

Next we let
$$
\hat H^2(\alpha) = \{ f \in H^2 : f \circ \gamma = \alpha(\gamma) f \},
$$
where $\alpha \in \Gamma^*$ and $H^2$ is the standard Hardy space in $\D$. As a consequence of the Widom condition \eqref{pyu3}, we see again that $\hat H^2(\alpha)$ is non-trivial.

The spaces $\check H^2(\alpha)$ in turn are defined via orthogonal complements. We note that the group $\Gamma$ acts on $\T$. Moreover, under the Widom condition, for this action there exists a measurable fundamental set \cite{Pom}. Thus, we can also define
$$
L^2(\alpha) = \{ f \in L^2 : f \circ \gamma = \alpha(\gamma) f \},
$$
analogously to the definitions above.

The annihilator of $\hat H^2_0(\alpha) := \{ f \in \hat H^2(\alpha) : f(0) = 0 \}$ (i.e., the orthogonal complement of this subspace with respect to the standard $L^2$-inner product on $L^2(\alpha)$) is contained in the subspace\footnote{To avoid confusion we emphasize that a $\overline{\cdot}$ over a complex function space indicates element-wise complex conjugation.}
$$
\Delta \overline{\hat H^2 (\nu \alpha^{-1})} := \{ \Delta \bar g : g \in \hat H^2 (\nu \alpha^{-1}) \} \subset L^2(\alpha).
$$
We let
$$
\check H^2(\alpha) := \{ g : \Delta \bar g \text{ belongs to the annihilator of } \hat H^2_0(\nu \alpha^{-1}) \}.
$$

For every $\alpha \in \Gamma^*$, by the Riesz representation theorem, there exist reproducing kernels $\hat k^\alpha, \check k^\alpha$ such that
\begin{align*}
\langle f, \hat k^\alpha \rangle & = f(0) \quad \text{ for every } f \in \hat H^2(\alpha), \\
\langle g, \check k^\alpha \rangle & = g(0) \quad \text{ for every } g \in \check H^2(\alpha),
\end{align*}
for which one can show the uniform bounds
\begin{equation}\label{e.repkernbounds}
\Delta^2(0) \le \check k^\alpha(0) \le \hat k^\alpha(0) \le 1.
\end{equation}
Normalizing these functions, we obtain
$$
\hat e^\alpha(\zeta) = \frac{\hat k^\alpha(\zeta)}{\|\hat k^\alpha\|}, \quad \check e^\alpha(\zeta) = \frac{\check k^\alpha(\zeta)}{\|\check k^\alpha\|}.
$$
It turns out that
$$
L^2(\alpha) = \Delta \overline{\check H^2_0 (\nu \alpha^{-1})} \oplus \{ \hat e^\alpha \} \oplus \hat H^2_0(\alpha) = \Delta \overline{\check H^2_0 (\nu \alpha^{-1})} \oplus \{ \Delta \overline{\check e^{\nu \alpha^{-1}}} \} \oplus \hat H^2_0(\alpha).
$$

The fundamental fact is (see the discussion at the end of \cite[Subsection~2.3]{VY14} and \cite{H83}):
\begin{equation}\label{e.dctcharacterization1}
\text{DCT holds } \quad \Leftrightarrow \quad \hat H^2(\alpha) = \check H^2(\alpha) \text{ for all } \alpha \in \Gamma^*.
\end{equation}

Denote
$$
\mathcal{NTF} = \{ \alpha \in \Gamma^* : \hat H^2(\alpha) \not= \check H^2(\alpha) \}.
$$
That is, $\mathcal{NTF}$ is precisely the set of $\alpha$'s that have a non-trivial fiber under $H^2(\alpha) \mapsto \alpha$. Then, by \eqref{e.dctcharacterization1},
\begin{equation}\label{e.dctcharacterization}
E \text{ satisfies DCT } \Leftrightarrow \mathcal{NTF} = \emptyset.
\end{equation}
Moreover, $\mathcal{NTF}$ is always small,
\begin{equation}\label{e.aentf}
\mathrm{Haar}_{\Gamma^*}(\mathcal{NTF}) = 0,
\end{equation}
where $\mathrm{Haar}_{\Gamma^*}$ denotes the Haar measure on $\Gamma^*$; see \cite[Theorem~6.2]{VY14}.

By an intermediate (character-automorphic) Hardy space we mean any closed space $H^2(\alpha)$ with
$$
\check H^2(\alpha) \subseteq H^2(\alpha) \subseteq \hat H^2(\alpha)
$$
and the invariance property
$$
\mathfrak{z} H^2_0(\alpha) \subseteq H^2(\alpha),
$$
where, as usual, $H^2_0(\alpha) := \{ f \in H^2(\alpha) : f(0) = 0 \}$. Let $k^\alpha$ be the reproducing kernel of this space and $e^\alpha := k^\alpha/\|k^\alpha\|$ its normalization. Thus,
\begin{equation}\label{e.h2alphadecomp}
H^2(\alpha) = \{ e^\alpha \} \oplus H^2_0(\alpha).
\end{equation}
Every function $f$ in $H^2_0(\alpha)$ is of the form $f = bg$ with $b$ from \eqref{e.bdefinition} and some $g \in \hat H^2(\alpha \mu^{-1})$. Thus, with $H^2(\alpha \mu^{-1})$ defined by $H^2_0(\alpha) = b H^2(\alpha \mu^{-1})$, \eqref{e.h2alphadecomp} can be rewritten as
$$
H^2(\alpha) = \{ e^\alpha \} \oplus b H^2(\alpha \mu^{-1}),
$$
and now we can iterate and obtain
$$
H^2(\alpha) = \{ e^\alpha \} \oplus \{ b e^{\alpha \mu^{-1}} \} \oplus b^2 H^2(\alpha \mu^{-2}) = ... = \{ e^\alpha \} \oplus \{ b e^{\alpha \mu^{-1}} \}\oplus \{ b^2 e^{\alpha \mu^{-2}} \}\oplus ...
$$
It is easy to see that the system
\begin{equation}\label{e.h2alphaonb}
e_n^\alpha(\zeta) := b^n(\zeta) \frac{k^{\alpha \mu^{-n}}(\zeta)}{\sqrt{k^{\alpha \mu^{-n}}(0)}}, \quad n \ge 0
\end{equation}
forms an orthonormal basis of $H^2(\alpha)$.

Next we shift in the other direction. Notice that by \eqref{e.repkernbounds}, we have $e_0^\alpha(0) > \check e^\alpha(0) > \Delta(0)$, so that $\mathfrak{z} e_0^\alpha$ does not belong to $H^2(\alpha)$ (since it has a simple pole at the origin). Thus, since $\mathfrak{z} e_0^\alpha$ is orthogonal to $b^2 H^2(\alpha \mu^{-2})$, there are $e_{-1}^\alpha$ and coefficients $a_0(\alpha), a_1(\alpha), b_0(\alpha)$ such that
$$
\mathfrak{z} e_0^\alpha = a_0(\alpha) e_{-1}^\alpha + b_0(\alpha) e_0^\alpha + a_1(\alpha) e_1^\alpha, \quad \| e_{-1}^\alpha \| = 1, \quad (b e_{-1}^\alpha)(0) > 0.
$$
In particular, $e_{-1}^\alpha$ is normalized and orthogonal to all $e_n^\alpha$, $n \ge 0$. We define $H^2(\alpha \mu)$ by
$$
b^{-1} H^2(\alpha \mu) = \{ e_{-1}^\alpha \} \oplus H^2(\alpha).
$$
Now we can iterate as before and obtain normalized vectors $e_n^\alpha$, $n \le -1$.

The system $\{ e_n^\alpha : n \in \Z \}$ forms an orthonormal basis of $L^2(\alpha)$ and the multiplication operator by $\mathfrak{z}$ in $L^2(\alpha)$ is given by a Jacobi matrix with respect to this basis. More precisely, we have the following result from \cite{VY14}.

\begin{prop}\label{p.fromh2alphatoJ}
With respect to the ONB $\{ e_n^\alpha : n \in \Z \}$, multiplication by $\mathfrak{z}$ in $L^2(\alpha)$ is given by the following Jacobi matrix $J = J(H^2(\alpha))$:
$$
\mathfrak{z} e_n^\alpha = a_n(\alpha) e_{n-1}^\alpha + b_n(\alpha) e_n^\alpha + a_{n+1}(\alpha) e_{n+1}^\alpha,
$$
where
\begin{align*}
a_n(\alpha) & = \mathcal{A}(\alpha \mu^{-n}) , \quad \mathcal{A}(\alpha) := (\mathfrak{z} b)(0) \sqrt{\frac{k^\alpha(0)}{k^{\alpha \mu}(0)}}, \\
b_n(\alpha) & = \mathcal{B}(\alpha \mu^{-n}) , \quad \mathcal{B}(\alpha) := \frac{(\mathfrak{z} b)(0)}{b'(0)} \left[ \frac{(k^\alpha)'(0)}{k^{\alpha}(0)} - \frac{(k^{\alpha \mu})'(0)}{k^{\alpha \mu}(0)} \right] + \frac{(\mathfrak{z} b)'(0)}{b'(0)}.
\end{align*}
\end{prop}

With the dual Hardy space
$$
b \tilde H^2(\mu^{-1} \alpha^{-1} \nu) := \Delta \overline{L^2(\alpha) \ominus H^2(\alpha)}
$$
and the dual basis
\begin{equation}\label{e.dualbasisdef}
b(\zeta) \tilde e_n^\alpha(\zeta) = \Delta(\zeta) \overline{e^\alpha_{-n-1}(\zeta)}, \quad \zeta \in \T,
\end{equation}
we obtain the Jacobi matrix $\tau J = J(\tilde H^2(\mu^{-1} \alpha^{-1} \nu))$, whose coefficients are related to those of $J = J(H^2(\alpha))$ by
$$
\tau a_n = a_{-n}, \quad \tau b_n = b_{-n-1}.
$$
The following result is \cite[Theorem~2.11]{VY14}.

\begin{prop}\label{e.Jisreflectionless}
The Jacobi matrix $J = J(H^2(\alpha))$ is reflectionless, that is, it belongs to $J(\tilde E)$. Moreover, we have
$$
r_+(\mathfrak{z}(\zeta)) = - \frac{e_0^\alpha(\zeta)}{a_0 e_{-1}^\alpha(\zeta)}, \quad r_-(\mathfrak{z}(\zeta)) = - \frac{\tilde e_0^\alpha(\zeta)}{a_0 \tilde e_{-1}^\alpha(\zeta)}.
$$
\end{prop}

\bigskip

Propositions~\ref{p.fromh2alphatoJ} and \ref{e.Jisreflectionless} establish a way of going from a space $H^2(\alpha)$ to a Jacobi matrix $J = J(H^2(\alpha)) \in J(\tilde E)$. Let us now discuss the other direction, that is, we start with a Jacobi matrix $J \in J(\tilde E)$ and wish to associate a space $H^2(\alpha)$.

The first step is to recall the description of $J(\tilde E)$ as a torus of dimension given by the number of gaps of $\tilde E$ (which is infinite in our present situation). Given a Jacobi matrix $J = J(\{ a_n, b_n \}) \in J(\tilde E)$ acting in $\ell^2(\Z)$, it is well known that the vectors $\{ \delta_{-1} , \delta_0 \}$ are cyclic for $J$. The matrix resolvent function is given by
$$
\begin{pmatrix} R_{-1,-1}(z) & R_{-1,0}(z) \\ R_{0,-1}(z) & R_{0,0}(z) \end{pmatrix} = \mathcal{E}^* (J-z)^{-1} \mathcal{E},
$$
where $\mathcal{E} : \C^2 \to \ell^2(\Z)$ is defined by
$$
\mathcal{E} \begin{pmatrix} c_{-1} \\ c_0 \end{pmatrix} = c_{-1} \delta_{-1} + c_0 \delta_0.
$$
The spectral theorem yields the existence of a $2 \times 2$ matrix measure $d\sigma$ such that
\begin{equation}\label{e.rzdsigmarelation}
R(z) = \int \frac{d\sigma(x)}{x - z}.
\end{equation}
With the half-line restrictions $J_\pm = P_{\ell^2(\Z_\pm)} J P_{\ell^2(\Z_\pm)}^*$ of $J$ and the associated Weyl-Titchmarsh functions
$$
r_+(z) := \langle (J_+-z)^{-1} \delta_0 , \delta_0 \rangle_{\ell^2(\Z_+)}, \quad r_-(z) := \langle (J_--z)^{-1} \delta_{-1} , \delta_{-1} \rangle_{\ell^2(\Z_-)}
$$
we have
$$
\begin{pmatrix} R_{-1,-1}(z) & R_{-1,0}(z) \\ R_{0,-1}(z) & R_{0,0}(z) \end{pmatrix} = \begin{pmatrix} r_-^{-1}(z) & a_0 \\ a_0 & r_+^{-1}(z) \end{pmatrix}^{-1},
$$
and in particular
\begin{align}
\label{e.resolventrplusminus1} - \frac{1}{R_{0,0}(z)} & = - \frac{1}{r_+(z)} + a_0^2 r_-(z), \\
\label{e.resolventrplusminus2} R_{-1,-1}(z) & = r_+(z)^{-1} r_-(z) R_{0,0}(z).
\end{align}
The function $R_{0,0}$ is clearly of Nevanlinna class, it is analytic in $\bar \C \setminus \tilde E$, and real in $\R \setminus \tilde E$. Since $J \in J(\tilde E)$, we have
\begin{equation}\label{e.jacobimatrixisreflectionless}
\frac{1}{r_+(x + i0)} = \overline{a_0^2 r_-(x + i0)} \quad \text{ for Lebesgue almost every}\ x \in \tilde E,
\end{equation}
and therefore $R_{0,0}$ is purely imaginary almost everywhere on $\tilde E$. This leads to the formula
\begin{equation}\label{e.reflectionlessresolvent}
R_{0,0}(z) = \frac{-1}{\sqrt{(z - z_0^-)(z - z_0^+)}} \prod_{k \ge 1} \frac{z - x_k}{\sqrt{(z - z_k^-)(z - z_k^+)}},
\end{equation}
where $x_k \in [z_k^-, z_k^+]$ is the unique point at which the argument of $R_{0,0}(x+i0)$ switches from $\pi$ to $0$. With this terminology, the space of divisors, $D(\tilde E)$, is defined as follows:
\begin{equation}\label{e.jacobidivisors}
D(\tilde E) = \{ (x_k, \varepsilon_k) : x_k \in [z_k^-, z_k^+], \; \varepsilon_k \in \{ 1, -1 \}, \; k \ge 1 \},
\end{equation}
with the identifications $(x_k^-,-1) = (x_k^-,1)$ and $(x_k^+,-1) = (x_k^+,1)$. Thus, $D(\tilde E) = \T^\infty$. The map from $J(\tilde E)$ to $D(\tilde E)$ is defined as follows. Given $J \in J(\tilde E)$, the $x_k$, $k \ge 1$, are defined by \eqref{e.reflectionlessresolvent}. The $\varepsilon_k$, $k \ge 1$, need only be defined if $x_k \in (z_k^-, z_k^+)$. We set $\varepsilon_k = 1$ (resp., $\varepsilon_k = -1$) if $x_k$ is a pole of $1/r_+$ (resp.,  a pole of $a_0^2 r_-$) (this is well defined since by \eqref{e.resolventrplusminus1} at least one of these functions has a pole at $x_k$ and by \eqref{e.resolventrplusminus2} and analyticity of $R_{-1,-1}$ at $x_k$ at most one of them has a pole at $x_k$). By our assumptions \eqref{e.econtainshalfline} and \eqref{e.discreteomegaks}, every $J \in J(\tilde E)$ has purely absolutely continuous spectrum, which is given by $\tilde E$, and hence by \cite[Lemma~3.2]{VY14}, the map $D : J(\tilde E) \to D(\tilde E)$ defined above is a bijection (and in fact a homeomorphism with respect to the strong operator topology on $J(\tilde E)$ and the product topology on the infinite product of circles $D(\tilde E)$).

The second ingredient is the construction of the generalized Abel map
\begin{equation}\label{e.genabelmap}
\tilde \pi : J(\tilde E) \to \Gamma^*.
\end{equation}
The following result is \cite[Lemma~4.5]{VY14}:

\begin{lemma}
Given $J = J(\{ a_n, b_n \}) \in J(\tilde E)$, there is a unique factorization
\begin{equation}\label{e.rpluse0e-1factorization}
r_+ \circ \mathfrak{z} = - \frac{1}{a_0} \frac{e_0}{e_{-1}}
\end{equation}
such that
\begin{align}\label{e.wronskianidentity}
a_0(e_{-1}(\zeta) & \overline{e_0(\zeta)} - e_0(\zeta) \overline{e_{-1}(\zeta)}) \\
\nonumber & = \sqrt{(\mathfrak{z}(\zeta) - z_0^-) (\mathfrak{z}(\zeta) - z_0^+)} \prod_{k \ge 1} \frac{\sqrt{(\mathfrak{z}(\zeta) - z_k^-) (\mathfrak{z}(\zeta) - z_k^+)}}{(\mathfrak{z}(\zeta) - c_k)}
\end{align}
for $\zeta \in \T$, where $e_0$ and $be_{-1}$ are of Smirnov class with mutually simple singular parts and $e_0(0) > 0$. Moreover, with the divisor $D(J) = \{ (x_k, \varepsilon_k) : k \ge 1 \}$ associated with $J$ and
$$
W(z) = \prod_{k \ge 1} \frac{z - x_k}{z - c_k},
$$
we have
\begin{equation}\label{e.ezerozetaformula}
e_0(\zeta) = \left( \prod_{k \ge 1} b_{x_k}^{(1 + \varepsilon_k)/2} \right) \sqrt{\frac{(W \circ \mathfrak{z})\Delta(\zeta)}{\prod_{k \ge 1} b_{x_k}(\zeta)}}.
\end{equation}
\end{lemma}

With the function $e_0$ from this lemma, the generalized Abel map $\tilde \pi$ in \eqref{e.genabelmap} is defined by sending $J$ to the character $\alpha \in \Gamma^*$ of $e_0$. With this $\alpha$, the next result, which is \cite[Theorem~4.10]{VY14}, associates an intermediate (character-automorphic) Hardy space $H^2(\alpha)$ with $J$ and this accomplishes our goal.

\begin{prop}\label{p.fromJtoH2alpha}
Suppose $J = J(\{ a_n, b_n \}) \in J(\tilde E)$ and $d\sigma$ is the $2 \times 2$ matrix measure from \eqref{e.rzdsigmarelation}. Then the map
$$
\begin{pmatrix} P_{-1}(x) \\ P_0(x) \end{pmatrix} \mapsto f(\zeta) := e_{-1}(\zeta) P_{-1}(\mathfrak{z}(\zeta)) + e_0(\zeta) P_0(\mathfrak{z}(\zeta))
$$
is unitary from $L^2(\R,\C^2;d\sigma)$ to $L^2(\alpha)$, with $\alpha = \tilde \pi(J)$. Moreover, the composition map
$$
\mathcal{F} : \ell^2(\Z) \to L^2(\R,\C^2;d\sigma) \to L^2(\alpha)
$$
is such that $H_J^2 := \mathcal{F}(P_{\ell^2(\Z_+)}^*(\ell^2(\Z_+)))$ has the properties
$$
\check H^2(\alpha) \subseteq H_J^2 \subseteq \hat H^2(\alpha), \quad \mathfrak{z}(H_J^2)_0 \subseteq H_J^2.
$$
\end{prop}

\bigskip

By Propositions~\ref{p.fromh2alphatoJ} and \ref{p.fromJtoH2alpha} we have maps $H^2(\alpha) \mapsto J(H^2(\alpha))$ and $J \mapsto H_J^2$. Volberg and Yuditskii also prove in \cite[Section~4]{VY14} that these maps are mutual inverses and hence we obtain a 1-1 correspondence between $J(\tilde E)$ and intermediate (character-automorphic) Hardy spaces $\{ H^2(\alpha) \}$.

Alternatively, $H^2(\alpha)$ can be defined via its reproducing kernel $k^\alpha(\zeta_1,\zeta_2)$. We note for later use that by \cite[Theorem~5.2]{VY14}, the reproducing kernel of $H^2(\alpha)$ is given by
\begin{equation}\label{e.h2alpharepkernelformula}
k^\alpha(\zeta_1,\zeta_2) = \frac{\begin{pmatrix} a_0 e_{-1}^\alpha(\zeta_1) & e_0^\alpha(\zeta_1) \end{pmatrix} \mathfrak{J} \begin{pmatrix} a_0 \overline{e_{-1}^\alpha(\zeta_2)} \\ \overline{e_0^\alpha(\zeta_2)} \end{pmatrix}}{\mathfrak{z}(\zeta_1) - \overline{\mathfrak{z}(\zeta_2)}},
\end{equation}
where
\begin{equation}\label{e.mathfrakJdef}
\mathfrak{J} = \begin{pmatrix} 0 & - 1 \\ 1 & 0 \end{pmatrix}.
\end{equation}

\subsection{Potentials in $q(E)$ and Character-Automorphic Hardy Spaces}

In this subsection we extend the bijection $J(\tilde E) \leftrightarrow \{ H^2(\alpha) \}$ discussed in the previous subsection to bijections
\begin{equation}\label{e.correspondence}
q(E) \leftrightarrow J(\tilde E) \leftrightarrow \{ H^2(\alpha) \}.
\end{equation}
Composing these two bijections, we obtain the desired bijection $q(E) \leftrightarrow \{ H^2(\alpha) \}$, which will be crucial to our study of the shift action on $q(E)$ in terms of the induced action on $\{ H^2(\alpha) \}$ in the next section.

\begin{prop}
There is a 1-1 correspondence between $q(E)$ and $J(\tilde E)$.
\end{prop}

\begin{proof}
Given $q \in q(E)$, construct $J$ as in Remark~\ref{r.fromqtoJ}. Since $q \in q(E)$, we have
$$
\lim_{\varepsilon \downarrow 0} m_+(\lambda + i \varepsilon) = - \lim_{\varepsilon \downarrow 0} \overline{m_-(\lambda + i \varepsilon)} \quad \text{for almost every } \lambda \in E.
$$
By \eqref{e.mplustorplus}--\eqref{e.mminustorminus}, this implies that
$$
\lim_{\varepsilon \downarrow 0} a_0^2 r_+(z + i \varepsilon) = - \lim_{\varepsilon \downarrow 0} \frac{1}{\overline{r_-(z + i \varepsilon)}} \quad \text{for almost every } z \in \tilde E.
$$
That is, $J \in J(\tilde E)$, and hence we have a map $q(E) \to J(\tilde E)$.

Conversely, suppose we are given $J \in J(\tilde E)$. Denote the associated half-line Weyl-Titchmarsh functions by $r_\pm$, that is, we have \eqref{e.rminusandJ}--\eqref{e.rplusandJ} and \eqref{e.jacobimatrixisreflectionless}. The function $r_+(-x^2)$ is analytic in the vicinity of the origin, and hence we have an expansion
$$
r_+(-x^2) = c_0 + c_1 x + c_2 x^2 + \cdots.
$$
Note that $c_1 < 0$. Then, the function $m_+$ defined by
$$
m_+(\lambda) = \frac{c_1}{c_0 - r_+(z)} - \frac{c_2}{c_1}
$$
 obeys
$$
m_+(\lambda) = -\sqrt{-\lambda} + o(1), \quad \lambda\to -\infty.
$$
Therefore, it is of the form
$$
m_+(\lambda) = - \sqrt{-1 - \lambda} \prod_{k \ge 1} \frac{\sqrt{(\lambda - \lambda_k^-)(\lambda - \lambda_k^+)}}{\lambda - \lambda_k} + \sum_{k \ge 1, \, \lambda_k \in (\lambda_k^-, \lambda_k^+)} \frac{\rho_k \varepsilon_k}{\lambda_k - \lambda},
$$
where
$$
\rho_j = \sqrt{-1 - \lambda_j} \sqrt{(\lambda_j - \lambda_j^-)(\lambda_j - \lambda_j^+)} \prod_{k \ge 1, \, k \not= j} \frac{\sqrt{(\lambda_j - \lambda_k^-)(\lambda_j - \lambda_k^+)}}{\lambda_j - \lambda_k}
$$
and $\varepsilon_k \in \{ \pm 1 \}$.

Standard inverse spectral theory yields the existence of a bounded continuous (using that $\R_+ \subset E$) half-line potential $q_+$ on $\R_+$. In the same
way we can produce $q_-$ on $\R_-$ from $r_-$. Moreover the left and right limits of $q_\pm$ coincide due to the standard trace formula,
$$
q_-(0-) = q_+(0+) = -1 + \sum_{k \ge 1} \left( \lambda_k^- + \lambda_k^+ - 2 \lambda_k \right).
$$
We obtain the potential $q$ whose restrictions to $\R_\pm$ are given by $q_\pm$ which belongs to the class $q(E)$ by construction. The collection of data $\{ (\lambda_k , \varepsilon_k) : k \ge 1 \}$ is called the divisor associated with the potential $q$.
\end{proof}

As a corollary of the construction above, we have the following analog of Proposition~\ref{p.fromJtoH2alpha}.

\begin{theorem}\label{t,spectral}
For every $q \in q(E)$, there exists a unique $H^2(\alpha)$ such that there exists a unitary map
$$
\mathcal{F} : L^2(\R) \to L^2(\alpha)
$$
with $\mathcal{F}(L^2(\R_+)) = H^2(\alpha)$ and
$$
\mathcal{F} L_q f = \left( \lambda_* + \frac{1}{\mathfrak{z}} \right) \mathcal{F} f
$$
for every $f \in D(L_q)$. Conversely, every $H^2(\alpha)$ arises in this way.
\end{theorem}

\begin{remark}\label{r.subspacechain}
Looking forward to the next section, one may ask about the image of $L^2(\R_{+,\ell}) = \{ f \in L^2(\R) : \mathrm{supp} (f) \subseteq [\ell,\infty) \}$ under the map $\mathcal{F}$. We will give the answer to this question in Theorem~\ref{t.subspacechain}.
\end{remark}

\begin{remark}\label{r.topologies}
Let us endow $q(E)$ with the compact-open topology (i.e., the topology of uniform convergence on compact subsets of $\R$), and similarly $J(\tilde E)$ with the topology of pointwise convergence of the coefficient sequences. The collection $D(E)$ of all divisors is equipped with the product topology, where each gap gives rise to a circle $\T$ (arising from $\lambda_k \in [\lambda_k^-, \lambda_k^+]$ and a choice of $\varepsilon_k = \pm 1$ when $\lambda_k \in (\lambda_k^-, \lambda_k^+)$), which carries the standard topology; compare \eqref{e.jacobidivisors}. All the maps between $q(E)$, $J(\tilde E)$, $D(E)$, $D(\tilde E)$ discussed above are continuous with respect to the topologies associated with them; compare, for example, Craig \cite{C89}.
\end{remark}

\section{The Action on Characters Induced by the Shift Action on $q(E)$}\label{s.4}

In this section we study the shift action on $q(E)$, that is,
$$
T_\ell : q(E) \to q(E), \quad q(\cdot) \mapsto q(\cdot - \ell).
$$
Since a potential $q$ is in 1-1 correspondence with its pair of Weyl-Titchmarsh functions $m_\pm = m_\pm(q)$, this suggests looking at $m_\pm(\ell,\cdot) = m_\pm(T_\ell (q))(\cdot)$. By the definition of the Weyl-Titchmarsh function (and the fact that the Jost solutions for the shifted potential are simply the shifted Jost solutions), we have
\begin{equation}\label{e.melldef}
u_\pm(\cdot,\ell,\lambda) = u_1(\cdot,\ell,\lambda) \pm m_\pm(\ell,\lambda) u_2(\cdot,\ell,\lambda),
\end{equation}
where $u_\pm(\cdot,\ell,\lambda), u_1(\cdot,\ell,\lambda), u_2(\cdot,\ell,\lambda)$ solve \eqref{e.eve}, and are subject to
$$
u_\pm(\cdot,\ell,\lambda) \text{ is square-integrable at } \pm \infty
$$
and
\begin{equation}\label{e.u12elldef}
\begin{pmatrix} u_1(\ell,\ell,\lambda) & u_2(\ell,\ell,\lambda) \\ u_1'(\ell,\ell,\lambda) & u_2'(\ell,\ell,\lambda) \end{pmatrix} = \begin{pmatrix} 1 & 0 \\ 0 & 1 \end{pmatrix}.
\end{equation}
In particular, we have $u_\pm(\cdot,0,\lambda) = u_\pm(\cdot,\lambda)$, $u_1(\cdot,0,\lambda) = u_1(\cdot,\lambda)$, and $u_2(\cdot,0,\lambda) = u_2(\cdot,\lambda)$.

Since the square-integrable solution (on either half-line) is unique up to a multiplicative constant, we have
$$
\begin{pmatrix} u_-(x,\ell,\lambda) & u_-'(x,\ell,\lambda) \\ u_+(x,\ell,\lambda) & u_+'(x,\ell,\lambda) \end{pmatrix} = \begin{pmatrix} c_-(\ell,\lambda) & 0 \\ 0 & c_+(\ell,\lambda) \end{pmatrix} \begin{pmatrix} u_-(x,\lambda) & u_-'(x,\lambda) \\ u_+(x,\lambda) & u_+'(x,\lambda)  \end{pmatrix}
$$
with suitable constants $c_\pm(\ell,\lambda)$. Evaluating this at $x = \ell$, and using \eqref{e.melldef} and \eqref{e.u12elldef}, we obtain
$$
\begin{pmatrix} 1 & - m_-(\ell,\lambda) \\ 1 & m_+(\ell,\lambda) \end{pmatrix} = \begin{pmatrix} c_-(\ell,\lambda) & 0 \\ 0 & c_+(\ell,\lambda) \end{pmatrix} \begin{pmatrix} 1 & - m_-(\lambda) \\ 1 & m_+(\lambda) \end{pmatrix} \begin{pmatrix} u_1(\ell,\lambda) & u_1'(\ell,\lambda) \\ u_2(\ell,\lambda) & u_2'(\ell,\lambda)  \end{pmatrix}
$$
Denote
$$
\mathfrak{A}(\ell,\lambda) := \begin{pmatrix} u_1(\ell,\lambda) & u_1'(\ell,\lambda) \\ u_2(\ell,\lambda) & u_2'(\ell,\lambda)  \end{pmatrix}.
$$
Then we can write
\begin{equation}\label{e.mplusellmplus}
\begin{pmatrix} 1 & m_+(\ell,\lambda) \end{pmatrix} \sim \begin{pmatrix} 1 & m_+(\lambda) \end{pmatrix} \mathfrak{A}(\ell,\lambda),
\end{equation}
where $v,w \in \C^2 \setminus \{ (0 \; 0) \}$ are equivalent, $v \sim w$, if they define the same element in $\PP \C^1$, that is, if there exists $c \in \C \setminus \{ 0 \}$ with $v = c w$.

On the other hand, due to \eqref{e.mplustorplus}, we have
\begin{equation}\label{e.mplusrplussimilarity1}
\begin{pmatrix} 1 & m_+(\ell,\lambda) \end{pmatrix} \sim \begin{pmatrix} 1 & - a_0^2(\ell) r_+(\ell,z) \end{pmatrix} C_\ell,
\end{equation}
where $C_\ell$ is a constant $\mathrm{SL}(2,\R)$ matrix and $a_0^2(\ell)$, $r_+(\ell,z)$ are the Jacobi matrix quantities associated with $T_\ell(q)$. Inserting a dilation, we can turn \eqref{e.mplusrplussimilarity1} into
\begin{equation}\label{e.mplusrplussimilarity}
\begin{pmatrix} 1 & m_+(\ell,\lambda) \end{pmatrix} \sim \begin{pmatrix} a_0(\ell) & - a_0(\ell) r_+(\ell,z) \end{pmatrix} \tilde C_\ell
\end{equation}
with a suitable constant $\mathrm{SL}(2,\R)$ matrix $\tilde C_\ell$.

Thus, it follows from \eqref{e.mplusellmplus} and \eqref{e.mplusrplussimilarity} that
\begin{equation}\label{e.rplusellrplussimilarity}
\begin{pmatrix} a_0(\ell) & - a_0(\ell) r_+(\ell,z) \end{pmatrix} \sim \begin{pmatrix} a_0 & - a_0 r_+(z) \end{pmatrix} \tilde C_0 \mathfrak{A}(\ell,\lambda) \tilde C_\ell^{-1}.
\end{equation}

Now recall the factorization \eqref{e.rpluse0e-1factorization}, that is, $r_+(\ell,\mathfrak{z}(\cdot)) = - \frac{1}{a_0(\ell)} \frac{e_0(\ell,\cdot)}{e_{-1}(\ell,\cdot)}$, and recall also that the generalized Abel map $\tilde \pi$ sends the Jacobi matrix corresponding to $T_\ell(q)$ to the character $\alpha_\ell \in \Gamma^*$ of $e_0(\ell,\cdot)$. Thus, to conform with our earlier notation we rewrite the factorization as $r_+(\ell,\mathfrak{z}(\cdot)) = - \frac{1}{a_0(\ell)} \frac{e_0^{\alpha_\ell}(\cdot)}{e_{-1}^{\alpha_\ell}(\cdot)}$. We can therefore rewrite \eqref{e.rplusellrplussimilarity} as follows: for every $\zeta \in \D$, there exists $s_\ell(\zeta) \in \C \setminus \{ 0 \}$\footnote{$s_\ell(\zeta)$ is always non-zero because $\mathfrak{A}(\ell,\lambda)$ is entire and has determinant one.} such that
\begin{equation}\label{e.alphaellflow}
s_\ell(\zeta) \begin{pmatrix} a_0(\ell) e_{-1}^{\alpha_\ell}(\zeta) & e_0^{\alpha_\ell}(\zeta) \end{pmatrix} = \begin{pmatrix} a_0 e_{-1}^{\alpha}(\zeta) & e_0^{\alpha}(\zeta) \end{pmatrix} \mathfrak{B}(\ell,\mathfrak{z}(\zeta)),
\end{equation}
where
$$
\mathfrak{B}(\ell,z) = \tilde C_0 \mathfrak{A}(\ell,\lambda) \tilde C_\ell^{-1}.
$$

We can rewrite \eqref{e.alphaellflow} as
$$
\frac{\Delta}{b} \overline{s_\ell \begin{pmatrix} a_0(\ell) e_{-1}^{\alpha_\ell} & e_0^{\alpha_\ell} \end{pmatrix}} = \frac{\Delta}{b} \overline{\begin{pmatrix} a_0 e_{-1}^{\alpha} & e_0^{\alpha} \end{pmatrix} \mathfrak{B}(\ell,\mathfrak{z})}
\ \text{on}\ \T,
$$
which is
\begin{equation}\label{e.alphaellflow2}
\frac 1{s_\ell} \begin{pmatrix} a_0(\ell) \tilde e_0^{\alpha_\ell} & \tilde e_{-1}^{\alpha_\ell} \end{pmatrix} = \begin{pmatrix} a_0 \tilde e_0^{\alpha} & \tilde e_{-1}^{\alpha} \end{pmatrix} \mathfrak{B}(\ell,\mathfrak{z}) \ \text{in}\ \D,
\end{equation}
due to \eqref{e.dualbasisdef} and the fact that $\overline{\mathfrak{B}(\ell,\overline{\mathfrak{z}})} = \mathfrak{B}(\ell,\mathfrak{z})$. Thus, with
$$
\Phi_\alpha = \begin{pmatrix} a_0 e_{-1}^\alpha & e_0^\alpha \\ a_0 \tilde e_0^\alpha & \tilde e_{-1}^\alpha \end{pmatrix},
$$
we see from \eqref{e.alphaellflow} and \eqref{e.alphaellflow2} that
\begin{equation}\label{e.Phialpha}
\begin{pmatrix} s_\ell & 0 \\ 0 & s_\ell^{-1} \end{pmatrix} \Phi_{\alpha_\ell} = \Phi_\alpha \mathfrak{B}(\ell,\mathfrak{z}).
\end{equation}

\begin{lemma}\label{l.Jcontraction}
With $\mathfrak{J}$ from \eqref{e.mathfrakJdef}, we have $\mathfrak{B}(\ell,z) \mathfrak{J} \mathfrak{B}(\ell,z)^* - \mathfrak{J} = 0$ if $z \in \R$ and
$$
\frac{\mathfrak{B}(\ell,z) \mathfrak{J} \mathfrak{B}(\ell,z)^* - \mathfrak{J}}{z - \bar z} \le 0
$$
if $z \in \C_-$.
\end{lemma}

\begin{proof}
This lemma is well known; compare, for example, \cite{AD08}. For the convenience of the reader we give the proof. Observe that
\begin{align*}
\mathfrak{A}(\ell,\lambda)' \mathfrak{J} & = \begin{pmatrix} u_1'(\ell,\lambda) & u_1''(\ell,\lambda) \\ u_2'(\ell,\lambda) & u_2''(\ell,\lambda)  \end{pmatrix} \begin{pmatrix} 0 & - 1 \\ 1 & 0 \end{pmatrix} \\
& = \begin{pmatrix} u_1'(\ell,\lambda) & (q(x) - \lambda) u_1(\ell,\lambda) \\ u_2'(\ell,\lambda) & (q(x) - \lambda) u_2(\ell,\lambda)  \end{pmatrix} \begin{pmatrix} 0 & - 1 \\ 1 & 0 \end{pmatrix} \\
& = \begin{pmatrix} u_1'(\ell,\lambda) & u_1''(\ell,\lambda) \\ u_2'(\ell,\lambda) & u_2''(\ell,\lambda)  \end{pmatrix} \begin{pmatrix} q(x) - \lambda & 0 \\ 0 & -1 \end{pmatrix} \\
& = \mathfrak{A}(\ell,\lambda) \left[ - \lambda \begin{pmatrix} 1 & 0 \\ 0 & 0 \end{pmatrix} + \begin{pmatrix} q(x) & 0 \\ 0 & -1 \end{pmatrix} \right] \\
& =: \mathfrak{A}(\ell,\lambda) \left[ - \lambda P + H \right].
\end{align*}
Notice that $P \ge 0$ and $H = H^*$.

Therefore,
\begin{align*}
(\mathfrak{A}(x,\lambda) \mathfrak{J} \mathfrak{A}(x,\lambda)^*)' & = \mathfrak{A}'(x,\lambda) \mathfrak{J} \mathfrak{A}(x,\lambda)^* + \mathfrak{A}(x,\lambda) \mathfrak{J} \mathfrak{A}'(x,\lambda)^* \\
& = \mathfrak{A}(x,\lambda) \left[ - \lambda P + H \right] \mathfrak{A}(x,\lambda)^* - \mathfrak{A}(x,\lambda) \left[ - \bar \lambda P + H \right] \mathfrak{A}(x,\lambda)^* \\
& = - \mathfrak{A}(x,\lambda) (\lambda - \bar\lambda) P \mathfrak{A}(x,\lambda)^*,
\end{align*}
and hence (since $\mathfrak{A}(0,\lambda)$ is the identity matrix)
\begin{align*}
0 & \le \int_0^\ell \mathfrak{A}(x,\lambda) P \mathfrak{A}(x,\lambda)^* \, dx \\
& = - \frac{\mathfrak{A}(\ell,\lambda) \mathfrak{J} \mathfrak{A}(\ell,\lambda)^* - \mathfrak{A}(0,\lambda) \mathfrak{J} \mathfrak{A}(0,\lambda)^*}{\lambda - \bar\lambda} \\
& = - \frac{\mathfrak{A}(\ell,\lambda) \mathfrak{J} \mathfrak{A}(\ell,\lambda)^* - \mathfrak{J}}{\lambda - \bar\lambda} .
\end{align*}
This shows that
$$
\frac{\mathfrak{A}(\ell,\lambda) \mathfrak{J} \mathfrak{A}(\ell,\lambda)^* - \mathfrak{J}}{\lambda - \bar\lambda}
$$
is non-positive for $\lambda$ in $\C_+$. Moreover, $\mathfrak{A}(\ell,\lambda) \mathfrak{J} \mathfrak{A}(\ell,\lambda)^* - \mathfrak{J}$ is zero for $\lambda \in \R$ since in this case $\mathfrak{A}(\ell,\lambda) \in \mathrm{SL}(2,\R)$. Since these properties are preserved after taking products with $\mathrm{SL}(2,\R)$ matrices, the claim for $\mathfrak{B}(\ell,z)$ follows (notice that if $\lambda \in \C_+$, then $z \in \C_-$).
\end{proof}

\begin{lemma}\label{l.AlessthanA'}
We have
$$
\left(  k^\alpha(\zeta_i, \zeta_j)  \right)_{1 \le i,j \le n} - \left(  k^{\alpha_\ell}(\zeta_i,\zeta_j) s_\ell(\zeta_i) \overline{s_\ell(\zeta_j)}  \right)_{1 \le i,j \le n} \ge 0
$$
\end{lemma}

\begin{proof}
Recall that due to \eqref{e.h2alpharepkernelformula}, we have
$$
k^\alpha(\zeta_1,\zeta_2) = \frac{\begin{pmatrix} a_0 e_{-1}^\alpha(\zeta_1) & e_0^\alpha(\zeta_1) \end{pmatrix} \mathfrak{J} \begin{pmatrix} a_0 \overline{e_{-1}^\alpha(\zeta_2)} \\ \overline{e_0^\alpha(\zeta_2)} \end{pmatrix}}{\mathfrak{z}(\zeta_1) - \overline{\mathfrak{z}(\zeta_2)}}
$$
Therefore, if we sandwich
$$
- \frac{\mathfrak{B}(\ell,\mathfrak{z}(\zeta_1)) \mathfrak{J} \mathfrak{B}(\ell,\mathfrak{z}(\zeta_2))^* - \mathfrak{J}}{\mathfrak{z}(\zeta_1) - \overline{\mathfrak{z}(\zeta_2)}}.
$$
between $\begin{pmatrix} a_0 e_{-1}^\alpha(\zeta_1) & e_0^\alpha(\zeta_1) \end{pmatrix}$ and $\begin{pmatrix} a_0 \overline{e_{-1}^\alpha(\zeta_2)} \\ \overline{e_0^\alpha(\zeta_2)} \end{pmatrix}$, we obtain a Hermite-positive function due to Lemma~\ref{l.Jcontraction} and \cite{AD08}. Combining this with \eqref{e.alphaellflow}, the lemma follows.
\end{proof}

\begin{lemma}\label{l.characterrelation}
For every $g \in H^2(\alpha_\ell)$, we have $s_\ell g \in H^2(\alpha)$.
\end{lemma}

\begin{proof}
For $m \ge 1$, choose $\zeta_{0,1}, \ldots, \zeta_{0,m} \in \D$ and $\eta_1, \ldots, \eta_m \in \C$ and consider the function $g \in H^2(\alpha_\ell)$ of the form
$$
g(\zeta) = \sum_{j = 1}^m \eta_j \overline{s_\ell(\zeta_{0,j})} k^{\alpha_\ell}(\zeta,\zeta_{0,j}).
$$
Our goal is to show that the function $f = s_\ell g$ belongs to $H^2(\alpha)$. To this end we will use \cite[Lemma~5.3]{VY14}. Thus we need to consider a finite collection $\zeta_1, \ldots, \zeta_n \in \D$ and prove the estimate
\begin{equation}\label{e.lemma53goal}
\left| \sum_{i = 1}^n f(\zeta_i) \bar \xi_i \right|^2 \le \|g\|^2_{H^2(\alpha_\ell)} \sum_{i,i' = 1}^n k^\alpha(\zeta_i,\zeta_{i'}) \bar \xi_i \xi_{i'}
\end{equation}
for all choices of $\xi_1, \ldots, \xi_n \in \C$. Once we accomplish this, \cite[Lemma~5.3]{VY14} implies that $f \in H^2(\alpha)$ and $\|f\|^2_{H^2(\alpha)} \le \|g\|^2_{H^2(\alpha_\ell)}$. Since the $g$'s of the form we consider are dense in $H^2(\alpha_\ell)$, the lemma then follows.

In order to prove \eqref{e.lemma53goal} we define the function $F : \D \times \D \to \C$ by
$$
F(\zeta_1, \zeta_2) := k^{\alpha_\ell}(\zeta_1,\zeta_2) s_\ell(\zeta_1) \overline{s_\ell(\zeta_2)}.
$$
This function is Hermite-positive, so that in particular the matrix
$$
\left(\begin{array}{ccccc|ccccc} F(\zeta_1, \zeta_1) && \cdots && F(\zeta_1, \zeta_n) & F(\zeta_1, \zeta_{0,1}) && \cdots && F(\zeta_1, \zeta_{0,m})  \\ &&&& & &&&& \\
\vdots &&&& \vdots & \vdots &&&& \vdots \\ &&&& & &&&& \\ F(\zeta_n, \zeta_1) && \cdots && F(\zeta_n, \zeta_n) & F(\zeta_n, \zeta_{0,1}) && \cdots && F(\zeta_n, \zeta_{0,m}) \\ \hline \\ F(\zeta_{0,1}, \zeta_1) && \cdots && F(\zeta_{0,1}, \zeta_n) & F(\zeta_{0,1}, \zeta_{0,1}) && \cdots && F(\zeta_{0,1}, \zeta_{0,m})   \\ &&&& & &&&& \\
\vdots &&&& \vdots & \vdots &&&& \vdots \\ &&&& & &&&& \\ F(\zeta_{0,m}, \zeta_1) && \cdots && F(\zeta_{0,m}, \zeta_n) & F(\zeta_{0,m}, \zeta_{0,1}) && \cdots && F(\zeta_{0,m}, \zeta_{0,m}) \end{array}\right)
$$
is positive definite. This matrix has block form
$$
\begin{pmatrix} A & B \\ B^* & D \end{pmatrix}
$$
and hence by replacing $A$ with
$$
A' = \begin{pmatrix} k^\alpha(\zeta_1, \zeta_1) & \cdots & k^\alpha(\zeta_1, \zeta_n) \\ \vdots && \vdots \\ k^\alpha(\zeta_n, \zeta_1) & \cdots & k^\alpha(\zeta_n, \zeta_n) \end{pmatrix},
$$
which is $\ge A$ by Lemma~\ref{l.AlessthanA'}, it follows that the matrix
$$
\begin{pmatrix} A' & B \\ B^* & D \end{pmatrix}
$$
is still positive definite. Therefore, with the vectors $\overrightarrow{\xi} = (\xi_1,\ldots,\xi_n)^T$ and $\overrightarrow{\eta} = (\eta_1, \ldots, \eta_m)^T$, we have
$$
\begin{pmatrix} \langle A' \overrightarrow{\xi} , \overrightarrow{\xi} \rangle & \langle B \overrightarrow{\eta}, \overrightarrow{\xi} \rangle \\ \langle \overrightarrow{\xi} , B \overrightarrow{\eta} \rangle & \langle D \overrightarrow{\eta}, \overrightarrow{\eta} \rangle \end{pmatrix} \ge 0.
$$
Thus,
\begin{equation}\label{e.lemma53goalcheck}
\left| \langle B \overrightarrow{\eta}, \overrightarrow{\xi} \rangle \right|^2 \le \langle D \overrightarrow{\eta}, \overrightarrow{\eta} \rangle \langle A' \overrightarrow{\xi} , \overrightarrow{\xi} \rangle.
\end{equation}
But
$$
\langle D \overrightarrow{\eta}, \overrightarrow{\eta} \rangle = \|g\|^2_{H^2(\alpha_\ell)},
$$
and hence \eqref{e.lemma53goalcheck} becomes \eqref{e.lemma53goal}.
\end{proof}

\begin{lemma}\label{l,sellisinner}
The function $s_\ell$ belongs to the Smirnov class $N_+$ and satisfies $|s_\ell| \le 1$ throughout $\D$ and $|s_\ell| = 1$ almost everywhere on $\T$. That is, $s_\ell$ is an inner function.
\end{lemma}

\begin{proof}
From \eqref{e.alphaellflow} we see that
\begin{equation}\label{e.sellratiorep}
s_\ell = \frac{\begin{pmatrix} a_0 e_{-1}^{\alpha}(\zeta) & e_0^{\alpha}(\zeta) \end{pmatrix} \mathfrak{B}(\ell,\mathfrak{z}(\zeta)) \begin{pmatrix} 1 \\ 0 \end{pmatrix}}{e_0^{\alpha_\ell}}.
\end{equation}
Since $s_\ell$ is analytic throughout $\D$, the zeros of the numerator in \eqref{e.sellratiorep} cancel out all zeros of the denominator. It follows that $s_\ell$ belongs to the Smirnov class $N_+$ and hence has boundary values almost everywhere on $\T$. By the maximum principle for the Smirnov class, the lemma therefore follows once we show that $|s_\ell| = 1$ almost everywhere on $\T$.

Again by \eqref{e.alphaellflow}, we have
\begin{align}
\label{e.sellinnerfunctionproof} |s_\ell(\zeta)|^2 & \begin{pmatrix} a_0(\ell) e_{-1}^{\alpha_\ell}(\zeta) & e_0^{\alpha_\ell}(\zeta) \end{pmatrix} \mathfrak{J} \begin{pmatrix} \overline{a_0(\ell) e_{-1}^{\alpha_\ell}(\zeta)} \\ \overline{e_0^{\alpha_\ell}(\zeta)} \end{pmatrix} \\
\nonumber & = \begin{pmatrix} a_0 e_{-1}^{\alpha}(\zeta) & e_0^{\alpha}(\zeta) \end{pmatrix} \mathfrak{B}(\ell,\mathfrak{z}(\zeta)) \mathfrak{J} \mathfrak{B}(\ell,\mathfrak{z}(\zeta))^* \begin{pmatrix} \overline{a_0 e_{-1}^{\alpha}(\zeta)} \\ \overline{e_0^{\alpha}(\zeta)} \end{pmatrix}
\end{align}
for every $\zeta \in \D$. Now send $\zeta$ to a point on the boundary of $\D$ where the boundary value of $s_\ell$ exists. Then, $\mathfrak{B}(\ell,\mathfrak{z}(\zeta)) \mathfrak{J} \mathfrak{B}(\ell,\mathfrak{z}(\zeta))^*$ will converge to $\mathfrak{J}$, and by the Wronskian identity \eqref{e.wronskianidentity} (note that the right-hand side of \eqref{e.wronskianidentity} is independent of the character $\alpha$!), we infer from \eqref{e.sellinnerfunctionproof} that $|s_\ell(\zeta)|^2$ tends to $1$.
\end{proof}

\begin{prop}\label{p.sellformula}
We have
\begin{equation}\label{e.formulaforsell}
s_\ell(\zeta) = \exp \left( i \ell \Theta_\infty \left( \lambda_* + \frac{1}{\mathfrak{z}(\zeta)} \right) \right).
\end{equation}
\end{prop}

\begin{proof}
First of all we note that each of the functions involved in the definition of $s_{\ell}(\zeta)$ can be extended by continuity to the boundary points $\zeta\in \T$ for which $\mathfrak{z}(\zeta)\in E\setminus\{0,-\frac{1}{\lambda_*}\}$ (the first exceptional point $0$ corresponds to the point of singularity of $\mathfrak{B}(\ell,z)$ and the second one is the accumulation point of the intervals which form the set $\tilde E$). In fact, $s_{\ell}(\zeta)$ has an analytic extension to the complement of the unit disk through each arc in the preimage of the interval $[z_{j+1}^+,z_j^-]$ in the fundamental domain of the group $\Gamma$. Therefore, due to Lemma~\ref{l,sellisinner}, $s_{\ell}(\zeta)$ is necessarily of the form
\begin{equation}\label{e.formulaforsell2}
s_{\ell}(\zeta)= \exp \left( i \ell_1 \Theta_0 \left( \lambda_* + \frac{1}{\mathfrak{z}(\zeta)} \right) \right) \exp \left( i \ell_2 \Theta_\infty \left( \lambda_* + \frac{1}{\mathfrak{z}(\zeta)} \right) \right),
\end{equation}
where $M_0(\lambda) = \Im \Theta_0(\lambda)$ and $M_\infty(\lambda) = \Im \Theta_\infty(\lambda)$ are the symmetric Martin functions in the domain $\mathcal{D}$, which correspond to the two possible singularities of the function $s_{\ell}(\zeta)$ on the boundary, that is, the origin and infinity.

The formula \eqref{e.formulaforsell2} becomes the asserted formula \eqref{e.formulaforsell} once we show that $\ell_1 = 0$ and $\ell_2 = \ell$. From \eqref{e.Phialpha} we have
$$
\frac{1}{s_\ell} = \frac{\begin{pmatrix} a_0(\ell) \tilde e_0^{\alpha_\ell} & \tilde e_{-1}^{\alpha_\ell} \end{pmatrix} \mathfrak B \begin{pmatrix} 1 \\ 0 \end{pmatrix}}{a_0(\ell) \tilde e_0^{\alpha}}.
$$
Since $\mathfrak B(\frac{1}{\lambda-\lambda_*})$ is uniformly bounded in the vicinity of $0$ and the functions $|\tilde e_0^{\alpha_\ell}|^2 \circ \mathfrak z^{-1}(\frac{1}{\lambda-\lambda_*})$ and $|(b\tilde e_{-1}^{\alpha_\ell})|^2\circ \mathfrak z^{-1}(\frac{1}{\lambda-\lambda_*})$ have harmonic majorants in $\Im \lambda>0$, we have
$$
\frac{1}{|s_\ell | \circ \mathfrak z^{-1}(\frac{1}{iy-\lambda_*})}\le \frac{C}{y}\frac{1}{{|\tilde e_0^{\alpha_\ell} | \circ \mathfrak z^{-1}(\frac{1}{iy-\lambda_*})}}.
$$
That is, if
$$
0<\ell_1=
\lim_{y\to +0}\frac{\log\frac{1}{|s_\ell | \circ \mathfrak z^{-1}(\frac{1}{iy-\lambda_*})}}{M_0(iy)}\le
\lim_{y\to +0}\frac{\log\frac{1}{|\tilde e_0^{\alpha_\ell} | \circ \mathfrak z^{-1}(\frac{1}{iy-\lambda_*})}}{M_0(iy)},
$$
then $\tilde e_0^{\alpha_\ell}$ has a non-trivial singular inner factor. But according to \eqref{e.ezerozetaformula} this is an outer function times, possibly, a Blaschke product. Thus $\ell_1=0$.

Finally, from \eqref{e.Phialpha} we have the two-sided estimate
$$
\frac{\|\mathfrak B(\ell,\mathfrak z(\zeta))\|}{\| \Phi_{\alpha}(\zeta)\| \cdot \| \Phi^{-1}_{\alpha_\ell}(\zeta)\|}
\le \frac{1}{|s_\ell(\zeta)|}\le \| \Phi_{\alpha}(\zeta)\| \cdot \| \Phi^{-1}_{\alpha_\ell}(\zeta)\| \cdot \|\mathfrak B(\ell,\mathfrak z(\zeta))\|.
$$
Since for all $H^2(\alpha)$
$$
\lim_{\lambda \to -\infty} \frac{\log \| \Phi^{\pm1}_\alpha(\mathfrak z^{-1}(\frac{1}{\lambda-\lambda_*})) \|}{\sqrt{|\lambda|}} = 0,
$$
and, as is well-known,
$$
\lim_{\lambda \to -\infty} \frac{\log \| \mathfrak A(\ell,\lambda) \|}{\sqrt{|\lambda|}} = \ell,
$$
due to the normalization \eqref{e.Thetanormalization}, we get $\ell_1 = \ell$.
\end{proof}

As an immediate corollary, we have the following result, which provides an answer to the question raised in Remark~\ref{r.subspacechain}.

\begin{theorem}\label{t.subspacechain}
For every $\ell \in \R$, we have
\begin{equation}\label{e.subspacechain}
\mathcal{F} \left( L^2(\R_{+,\ell}) \right) = \exp \left( i \ell \Theta_\infty \left( \lambda_* + \frac{1}{\mathfrak{z}} \right) \right) H^2(\alpha \chi_\ell^{-1}),
\end{equation}
where $\mathcal{F}$ is the map from Theorem~\ref{t,spectral}.

In particular, the chain of subspaces $\{ \mathcal{F} \left( L^2(\R_{+,\ell}) \right) \}_{\ell \in \R}$ is transformed into the chain of subspaces $\left\{ \exp \left( i \ell \Theta_\infty \left( \lambda_* + \frac{1}{\mathfrak{z}} \right) \right) H^2(\alpha \chi_\ell^{-1}) \right\}_{\ell \in \R}$.
\end{theorem}

\begin{remark}
(a) On the right-hand side of \eqref{e.subspacechain}, we do not specify which intermediate Hardy space $H^2(\alpha \chi_\ell^{-1})$ arises in situations where $\alpha \chi_\ell^{-1} \in \mathcal{NTF}$.

(b) Theorem~\ref{t.subspacechain} should be compared with the classical Payley-Wiener result stating that
$$
F \left( L^2(\R_{+,\ell}) \right) = \exp \left( i \ell z \right) H^2
$$
for every $\ell \in \R$, where $F$ denotes the Fourier transform on $L^2(\R)$.
\end{remark}

\begin{prop}\label{p.flowrelation}
Via the composition of the maps
$$
q(E) \to J(\tilde E) \to \Gamma^*, \quad q \mapsto J \mapsto \alpha,
$$
the push-forward of the flow $T_\ell$ on $q(E)$ is the flow $\mathcal{S}_\ell$ on $\Gamma^*$. In particular, the push-forward is minimal and uniquely ergodic with respect to the Haar measure on $\Gamma^*$.
\end{prop}

\begin{proof}
Recall from \eqref{e.chielldef} that $\chi_\ell$ is the character of $s_\ell$. Recall also that the generalized Abel map $\tilde \pi$ sends the Jacobi matrix corresponding to $T_\ell(q)$ to the character $\alpha_\ell \in \Gamma^*$, that is, the flow $\ell \mapsto \alpha_\ell$ is precisely the one associated with the translation flow on $q(E)$ with initial condition $q$. From Lemma~\ref{l.characterrelation} we see that $\chi_\ell \alpha_\ell = \alpha$, that is,
\begin{equation}\label{e.alphaellchiell}
\alpha_\ell = \chi_\ell^{-1} \alpha.
\end{equation}
Combining \eqref{e.chielldef} and \eqref{e.alphaellchiell}, the first statement of the proposition follows. The second statement follows from the first statement along with the assumption \eqref{e.frequenciesindependent}.
\end{proof}

\begin{remark}\label{r.measures}
The map $\pi : q(E) \to \Gamma^*$, $q \mapsto \alpha$ from Proposition~\ref{p.flowrelation} is called the generalized Abel map. It is continuous with respect to the compact-open topology on $q(E)$. Since $q(E)$ is compact in the compact-open topology, it admits a translation invariant probability measure $dq$. By Proposition~\ref{p.flowrelation} the push-forward of this measure, $\pi_*(dq)$, is a translation invariant probability measure on $\Gamma^*$. Due to assumption \eqref{e.frequenciesindependent}, $\pi_*(dq)$ is equal to normalized Haar measure on $\Gamma^*$. By \eqref{e.aentf}, the generalized Abel map $\pi$ is almost everywhere 1-1. Thus, the measure $dq$ is actually uniquely determined by $\pi$ and Haar measure on $\Gamma^*$. In particular, $q(E)$ is uniquely ergodic with respect to the shift action. The corresponding measure on $D(E)$ (namely, the push-forward of the measure $dq$ under the map $q(E) \to D(E)$ discussed in Remark~\ref{r.topologies}) was shown in \cite[Section~7]{VY14} to assign positive measure to any non-empty open set. Pulling back this property, we find that $dq$ assigns positive weight to any non-empty open set in $q(E)$.
\end{remark}

\section{Proof of the Main Theorem and Corollary}\label{s.5}

In this section we prove Theorem~\ref{t.main} and Corollary~\ref{c.main}. Given our results up to this point, the proofs of these statements are analogous to the proofs given in \cite{VY14} of the corresponding statements in the Jacobi matrix case. We provide the details for the convenience of the reader.

\begin{proof}[Proof of Theorem~\ref{t.main}]
(a) Assume that $E$ satisfies DCT. Then, by \eqref{e.dctcharacterization}, we have $\mathcal{NTF} = \emptyset$. Thus, the map between $q \in q(E)$ and $\alpha \in \Gamma^*$ is a homeomorphism; compare, for example, \cite{SY95}. Since the translation flow on $q(E)$ corresponds to a strictly ergodic translation flow on the compact Abelian group $\Gamma^*$ via this homeomorphism, it follows that every $q \in q(E)$ is almost periodic.

(b) Assume that $E$ does not satisfy DCT. We have to show that every $q \in q(E)$ is not almost periodic. We will proceed in three steps. In the first step, we show that there exists a $q \in q(E)$ that is not almost periodic. In the second step we show that almost every $q \in q(E)$ is not almost periodic. In the third and final step, we then show the result in complete generality, that is, indeed every $q \in q(E)$ is not almost periodic. The reason for breaking up the proof into three steps is that each step relies on the previous one.

Let us begin with the first step. As $E$ does not satisfy DCT, by \eqref{e.dctcharacterization}, we have $\mathcal{NTF} \not= \emptyset$. Choose any $\alpha \in \mathcal{NTF}$ and consider the (distinct!) potentials $q(\hat H^2(\alpha)), q(\check H^2(\alpha)) \in q(E)$ corresponding to the spaces $\check H^2(\alpha) \subsetneq \hat H^2(\alpha)$. We claim that at least one of these potentials is not almost periodic. Assume to the contrary that both are almost periodic. By Proposition~\ref{p.flowrelation}, \eqref{e.frequenciesindependent}, and \eqref{e.aentf} it is possible to choose $\beta \in \Gamma^* \setminus \mathcal{NTF}$ and a sequence $\ell_j \to \infty$ such that $\mathcal{S}_{\ell_j} \alpha \to \beta$ as $j \to \infty$, and $\mathcal{S}_{\ell_j} \alpha$ is the character associated with both $T_{\ell_j} q(\hat H^2(\alpha))$ and $T_{\ell_j} q(\check H^2(\alpha))$. Since $\beta \in \Gamma^* \setminus \mathcal{NTF}$, there exists a unique $q_\beta \in q(E)$ associated with this character. By the almost periodicity of $q(\hat H^2(\alpha)), q(\check H^2(\alpha))$ and the choices above, we can (twice) pass to a subsequence of $\{ \ell_j \}$, which we still denote by $\{ \ell_j \}$, such that
$$
\lim_{j \to \infty} \| T_{\ell_j} q(\hat H^2(\alpha)) - q_\beta \|_\infty = \lim_{j \to \infty} \| T_{\ell_j} q(\check H^2(\alpha)) - q_\beta \|_\infty = 0.
$$
But this is a contradiction since
\begin{align*}
0 & < \| q(\hat H^2(\alpha)) - q(\check H^2(\alpha)) \|_\infty \\
& = \| T_{\ell_j} q(\hat H^2(\alpha)) - T_{\ell_j} q(\check H^2(\alpha)) \|_\infty \\
& \le \| T_{\ell_j} q(\hat H^2(\alpha)) - q_\beta \|_\infty + \| q_\beta - T_{\ell_j} q(\check H^2(\alpha)) \|_\infty \\
& \to 0 \text{ as } j \to \infty.
\end{align*}
This completes the first step.

For the second step, fix some $q_0 \in q(E)$ that is not almost periodic. By the first step, such a $q_0$ exists. Choose open (in the compact-open topology) sets $O_n$, $n \ge 1$, in $q(E)$ with
$$
O_1 \supset O_2 \supset \cdots \supset O_n \supset \cdots, \quad \bigcap_{n \ge 1} O_n = \{ q_0 \}.
$$
By Remark~\ref{r.measures}, the measure $dq$ on $q(E)$ assigns strictly positive weight to each $O_n$. Thus, if we consider the shift-invariant set
$$
V_n = \bigcup_{\ell \in \R} T_\ell(O_n),
$$
it follows from ergodicity of $dq$ that $dq$ assigns full weight to each $V_n$, $n \ge 1$. As a consequence, the set
$$
V := \bigcap_{n \ge 1} V_n
$$
has full $dq$ measure. We claim that every $q \in V$ is not almost periodic. Suppose to the contrary that there exists $q_1 \in V$ that is almost periodic. By construction of the set $V$, $q_0$ is an accumulation point (in the compact-open topology) of the set of translates of the almost periodic function $q_1$, which in turn implies that $q_0$ is almost periodic as well;\footnote{If $q_0$ is the limit of $T_{\ell_k} q_1$ with respect to uniform convergence on compact subsets, by almost periodicity of $q_1$ we can choose a subsequence of $\{ T_{\ell_k} q_1 \}$ that converges uniformly to an almost periodic limit. This limit must of course also be the limit in the compact-open topology and hence be equal to $q_0$. This implies that $q_0$ is almost periodic.} contradiction. This shows that every $q$ in the full measure subset $V$ of $q(E)$ is not almost periodic, and this completes the second step.

For the third step, let $q \in q(E)$ be arbitrary and assume that $q$ is almost periodic. Then any accumulation point (in the compact-open topology) of translates of $q$ must be almost periodic as well. Let us force accumulation on a point that is already known to be not almost periodic in order to get a contradiction. With the generalized Abel map $\pi : q(E) \to \Gamma^*$, let $\alpha = \pi(q)$. By Proposition~\ref{p.flowrelation}, we have $\pi(T_\ell q) = \mathcal{S}_\ell \alpha$. Recall from \eqref{e.frequenciesindependent} that $\{\mathcal S_\ell\alpha : \ell\in\R\}$ is dense. Since both $\pi(V)$ and $\Gamma^* \setminus \mathcal{NTF}$ have full Haar measure, their intersection has full Haar measure, and in particular it is not empty. Thus, let us choose $\beta$ in this intersection, and then $\{ \ell_n \} \subset \R$ with $\mathcal S_{\ell_n} \alpha \to \beta$. Now, since $\beta \in \Gamma^* \setminus \mathcal{NTF}$, there is a unique $\tilde q \in q(E)$ with $\pi(\tilde q) = \beta$, and since $\beta \in \pi(V)$, $\tilde q$ is not almost periodic by the previous step. Since $T_{\ell_n} q \to \tilde q$ by construction, we have accomplished what we wanted, namely, we have found a non-almost periodic accumulation point of translates of an almost periodic function, which is impossible. This contradiction shows that $q$ is indeed not almost periodic, and this completes the third step and the proof of the theorem.
\end{proof}

\begin{proof}[Proof of Corollary~\ref{c.main}]
Consider a comb domain $\Pi(\{\omega_k,h_k\}_{k=1}^\infty)$ subject to the assumptions \eqref{e.discreteomegaks}--\eqref{e.frequenciesindependent}. Suppose $E$ is the image of $\R_+$ under {\rm (}the continuous extension to the closure of{\rm )} a conformal map sending the comb domain to the upper half-plane. The corollary follows from Theorem~\ref{t.main} as soon as it is established that under these assumptions it is possible that the set $E$ does not satisfy DCT. One can use, for example, conformal images of the sets discussed in \cite[Proposition~1.10]{VY14} (note that the condition \eqref{e.frequenciesindependent} can indeed be ensured for appropriate choices of gap boundaries within the framework of \cite[Proposition~1.10]{VY14}), which were shown there to satisfy the Widom condition but not DCT. Of course, the Widom condition and DCT are conformally invariant. This concludes the proof.
\end{proof}

\begin{appendix}

\section{The Kotani-Last Conjecture for Extended CMV Matrices}

In this appendix we discuss the Kotani-Last conjecture for extended CMV matrices and how to disprove it by following the strategy we employed in the main body of the paper for continuum Schr\"odinger operators.

CMV matrices arise in the study of orthogonal polynomials on the unit circle, and they are the canonical matrix representation of a unitary operator on a separable Hilbert space with a cyclic vector. By the spectral theorem, the latter operator is unitarily equivalent to multiplication by the independent variable in $L^2(\T,d\nu)$. Applying the Gram-Schmidt orthonormalization procedure to $1,z,z^{-1},z^2,z^{-2},\ldots$, we obtain a basis of $L^2(\T,d\nu)$ with respect to which the operator is represented by the matrix
$$
\mathcal{C} = \begin{pmatrix}
{}& \bar\upsilon_0 & \bar\upsilon_1 \rho_0 & \rho_1
\rho_0
& 0 & 0 & \dots & {} \\
{}& \rho_0 & -\bar\upsilon_1 \upsilon_0 & -\rho_1
\upsilon_0
& 0 & 0 & \dots & {} \\
{}& 0 & \bar\upsilon_2 \rho_1 & -\bar\upsilon_2 \upsilon_1 &
\bar\upsilon_3 \rho_2 & \rho_3 \rho_2 & \dots & {} \\
{}& 0 & \rho_2 \rho_1 & -\rho_2 \upsilon_1 &
-\bar\upsilon_3
\upsilon_2 & -\rho_3 \upsilon_2 & \dots & {} \\
{}& 0 & 0 & 0 & \bar\upsilon_4 \rho_3 & -\bar\upsilon_4
\upsilon_3
& \dots & {} \\
{}& \dots & \dots & \dots & \dots & \dots & \dots & {}
\end{pmatrix},
$$
where $\upsilon_n \in \D = \{ w \in \C : |w| < 1 \}$ and $\rho_n = (1-|\upsilon_n|^2)^{1/2}$. The parameters $\{ \upsilon_n \}_{n \in \Z_+}$ are called Verblunsky coefficients. A matrix $\mathcal{C}$ of this form is called a CMV matrix. We refer the reader to \cite{S05, S05b} for background on orthogonal polynomials on the unit circle and CMV matrices.

The matrix $\mathcal{C}$ acts in $\ell^2(\Z_+)$. Of course, there is a natural two-sided extension of $\mathcal{C}$, which acts in $\ell^2(\Z)$. Such a two-sided infinite matrix is determined by Verblunsky coefficients $\{ \upsilon_n \}_{n \in \Z}$, it is denoted by $\mathcal{E}$, and it is called an extended CMV matrix.

The two-sided situation is natural if the coefficients are defined by suitable sampling along the orbits of an invertible map. To fix notation, let $\Omega$ be a compact metric space, $T : \Omega \to \Omega$ a homeomorphism, $d\mu$ a $T$-ergodic Borel probability measure on $\Omega$, and $f : \Omega \to \D$ measurable. Then, we have coefficient sequences
$$
\upsilon_n(\omega) = f(T^n \omega), \quad \omega \in \Omega, \; n \in \Z.
$$
The extended CMV matrix corresponding to the sequence $\upsilon(\omega) = \{ \upsilon_n(\omega) \}_{n \in \Z}$ is denoted by $\mathcal{E}_\omega$. As in the Schr\"odinger case, we have the following general result. There are sets $\Sigma, \Sigma_\mathrm{ac}, \Sigma_\mathrm{sc}, \Sigma_\mathrm{pp} \subseteq \T$ and a set $\Omega_0 \subseteq \Omega$ with $\mu(\Omega_0) = 1$ such that $\sigma(\mathcal{E}_\omega) = \Sigma$ and $\sigma_\bullet(\mathcal{E}_\omega) = \Sigma_\bullet$, $\bullet \in \{ \mathrm{ac}, \mathrm{sc}, \mathrm{pp} \}$, for every $\omega \in \Omega_0$.

In this setting, the Kotani-Last conjecture takes the same form as before:

\begin{klc}
If the family $\{ \mathcal{E}_\omega \}_{\omega \in \Omega}$ is such that $\Sigma_\mathrm{ac} \not= \emptyset$, then the sequences $\{ \upsilon(\omega) \}_{\omega \in \Omega}$ are almost periodic.
\end{klc}

As before this means that for each $\upsilon(\omega)$, the set of its translates is relatively compact in $\ell^\infty(\Z)$ or, equivalently, that we can choose $\Omega$ to be a compact abelian group, $T$ a minimal translation, $\mu$ Haar measure, and $f$ continuous. Slightly abusing terminology, we will say that an extended CMV matrix $\mathcal{E}$ is almost periodic if its Verblunsky coefficients form an almost periodic sequence.

Our goal is to show that the Kotani-Last conjecture for extended CMV matrices fails:

\begin{theorem}\label{tecmvklcfails}
There are families $\{ \mathcal{E}_\omega \}_{\omega \in \Omega}$ of the form above such that for every $\omega \in \Omega$, the extended CMV matrix $\mathcal{E}_\omega$ has purely absolutely continuous spectrum and is not almost periodic.
\end{theorem}

As mentioned already, the strategy will be the same as the one employed above for continuum Schr\"odinger operators. Thus, consider a set $E_2 \subset \T$ in the complex ($\varphi$-) plane. We write
\begin{equation}\label{e.e2def}
\T \setminus E_2 = \bigcup_{k = 0}^\infty (\varphi_k^-,\varphi_k^+),
\end{equation}
where the arc $(\varphi_k^-,\varphi_k^+) \subset \T$ represents the $k$-th gap of $E_2$. Let $\mathcal{D}_2 := \bar \C \setminus E_2$.

Now let us consider the extended CMV matrices we want to associate with the set $E_2$. We will require that their spectrum is equal to $E_2$ and that they are reflectionless on $E_2$. To make the latter condition explicit, let us recall that we can associate a Schur function $s_+$ with the coefficients $\{ \upsilon_n \}_{n \in \Z_+}$ so that
$$
s_+(\varphi) = \frac{\upsilon_0 + \varphi s^{(1)}_+(\varphi)}{1 + \varphi \bar \upsilon_0 s_+^{(1)}(\varphi)}, \; \upsilon_0 = s_+(0), \quad s_+^{(1)}(\varphi) = \frac{\upsilon_1 + \varphi s^{(2)}_+(\varphi)}{1 + \varphi \bar \upsilon_1 s_+^{(2)}(\varphi)}, \; \upsilon_1 = s_+^{(1)}(0), \ldots .
$$
Similarly, we can associate a Schur function $s_-$ with the coefficients $\{ \upsilon_n \}_{n \in \Z_-}$. The extended CMV matrix corresponding to the Verblunsky coefficients $\{ \upsilon_n \}_{n \in \Z}$ is called reflectionless on $E_2$ if we have $\overline{\varphi s_+(\varphi)} = s_-(\varphi)$ Lebesgue almost everywhere on $E_2$. We will denote the set of extended CMV matrices which are reflectionless on $E_2$ and have spectrum $E_2$ by $\mathcal{E}(E_2)$. We equip the set $\mathcal{E}(E_2)$ with the topology of strong operator convergence, which on the level of Verblunsky sequences means that they converge pointwise.

If all reflectionless measures on $E_2$ are absolutely continuous, the set $\mathcal{E}(E_2)$ may be parametrized, as usual, in terms of divisors $D \in D(E_2)$. Each gap $(\varphi_k^-,\varphi_k^+)$ of $E_2$ corresponds to a circle, namely $\{ (\varphi_k , \varepsilon_k) : \varphi_k \in [\varphi_k^-,\varphi_k^+], \varepsilon_k = \pm 1 \}$ with the identifications $(\varphi_k^-,-1) = (\varphi_k^-,1)$ and $(\varphi_k^+,-1) = (\varphi_k^+,1)$, and with the standard topology on this circle. The space $D(E_2)$ of divisors $\{ (\varphi_k , \varepsilon_k) : k \in \Z_+ \}$ is the product of these circles equipped with the product topology. It is homeomorphic to $\mathcal{E}(E_2)$.

As before we will use uniformization for the given domain and elements of potential theory and character-automorphic functions. For the uniformization we can use the same map $\mathfrak{z}$ as before once we apply a suitable fractional linear transformation. Choose $\varphi_* \in (\varphi_0^-,\varphi_0^+)$ and define the new variable
\begin{equation}\label{e.efromvtoz}
z = i \frac{\varphi_* + \varphi}{\varphi_* - \varphi}.
\end{equation}
Its inverse is given by
\begin{equation}\label{e.efromztov}
\varphi = \varphi_* \frac{z - i}{z + i}.
\end{equation}

Note that \eqref{e.efromvtoz} maps $\varphi_* \mapsto \infty$, $0 \mapsto i$, $\infty \mapsto - i$, and it maps the set $E_2$ to the compact set
$$
\tilde E_2 = \left\{ i \frac{\varphi_* + \varphi}{\varphi_* - \varphi} : \varphi \in E_2 \right\}
$$
in the $z$-plane. We consider a uniformization  $\D / \Gamma \simeq \tilde{\mathcal{D}}_2$ of $\tilde{\mathcal{D}}_2 := \bar \C \setminus \tilde E_2$ with  a Fuchsian group $\Gamma$ and a meromorphic function $\mathfrak{z} : \D \to \tilde{\mathcal{D}}_2$ with the same properties as before. There is $\kappa \in \D \cap \C_-$ with $\mathfrak{z}(\kappa) = i$ and $\mathfrak{z}(\bar \kappa) = - i$. Composing the maps, we obtain the uniformization $\D / \Gamma \simeq \mathcal{D}_2$ via $\varphi_* \frac{\mathfrak{z} - i}{\mathfrak{z} + i} : \D \to \mathcal{D}_2$.

With the Green function $b_{z_0}$ of the group $\Gamma$ from before, we have
\begin{equation}\label{e.bibminusi}
\mathfrak{v} := \varphi_* \frac{\mathfrak{z} - i}{\mathfrak{z} + i} = e^{i \psi_1} \frac{b_i}{b_{-i}},
\quad b_i(\bar\kappa)>0,\ b_{-i}(\kappa)>0,
\end{equation}
for some $\psi_1 \in [0,2\pi)$. Note that $\mathfrak{v}$ is automorphic, that is, $\mathfrak{v} \circ \gamma = \mathfrak{v}$ for every $\gamma \in \Gamma$. Due to \eqref{e.bibminusi} this implies that the characters of $b_i$ and $b_{-i}$ coincide; we still denote this common character by $\mu_i$.

In order to describe the character $\mu_i$ explicitly, we parametrize the set $E_2$, up to a rotation, using a periodic comb domain. Let $\Pi_2 = \Pi_2 (\{\omega_k, h_k\}_{k=1}^\infty)$,
\begin{equation}\label{e.Thetadefinition2}
\Pi_2 := \{ w \in \C : \Im w > 0 \} \setminus \{ \omega_k + 2 j \pi + i y : 0 < y \le h_k : k \in \Z_+, \; j \in \Z \}
\end{equation}
with frequencies $\{ \omega_k \}_{k \in \Z_+} \subset [0,2\pi)$, $\omega_0 < \omega_1 < \dots < \omega_k < \omega_{k+1} < \dots$, and heights $\{ h_k \}_{k \in \Z_+} \subset (0, \infty)$ such that
\begin{equation}\label{e.discreteomegaks2}
\lim_{k \to \infty} \omega_k =: \omega_*
\end{equation}
and
\begin{equation}\label{e.widomcondition2}
\sum_{k \ge 1} h_k < \infty.
\end{equation}

\begin{lemma}\label{l.periodictheta}
Let $\Pi_2$ be a periodic comb domain as in \eqref{e.Thetadefinition2}. Let $\Theta_2$ be a conformal mapping from $\C_+$ onto $\Pi_2$, which is normalized by
$$
\Theta_2(iy) = iy + \dots, \quad y \to \infty.
$$
Then $\Theta_2$ is $2\pi$ periodic, that is, $\Theta_2 (\psi + 2\pi) = \Theta_2(\psi) + 2 \pi$ for every $\psi \in \C_+$.
\end{lemma}

\begin{proof}
The conformal mapping $\Theta_2^{-1}(\Theta_2(\psi) + 2\pi)$ makes sense and maps $\C_+$ onto itself. Therefore by the normalization condition, $\Theta_2^{-1} (\Theta_2(\psi) + 2 \pi) = \psi + x_0$. Our claim is $x_0 = 2\pi$.

We use the integral representation for $\Theta_2(\psi)$,
$$
\frac{\Theta_2(\psi + x_0) - \Theta_2(\psi)}{x_0} = 1 + \frac{1}{\pi} \int_{\R} \frac{\rho(x) \, dx}{(x-\psi) (x-(x_0+\psi))},
$$
where $\rho(x) = \Im \Theta_2(x)$. Note that $\rho(x)$ is uniformly bounded. Since $\Theta_2(\psi + x_0) - \Theta_2(\psi) = 2\pi$, we get
$$
\frac{1}{\pi} \int_{\R} \frac{\rho(x) \, dx}{(x-\psi) (x-(x_0+\psi))} = \frac{2\pi}{x_0} - 1.
$$
We put $\psi = iy$ and pass to the limit as $y \to \infty$. Our goal is to show that the limit of the integral is zero.

We separate real and imaginary parts,
$$
\frac{1}{\pi} \int_{\R} \frac{x - x_0 + iy}{(x - x_0)^2 + y^2} \frac{x + iy}{x^2 + y^2} \rho(x) \, dx = I_1 + i I_2,
$$
where
$$
I_1 = \frac{1}{\pi} \int_{\R} \frac{x (x - x_0) - y^2}{((x - x_0)^2 + y^2)(x^2 + y^2)} \rho(x) \, dx
$$
and
$$
I_2 = \frac{1}{\pi} \int_{\R} \frac{y(2x - x_0)}{((x - x_0)^2 + y^2)(x^2 + y^2)} \rho(x) \, dx.
$$
Taking into account $I_2 = 0$, we get
$$
I_1 = \frac{1}{\pi} \int_{\R} \frac{x^2 - x_0^2/2 - y^2}{((x - x_0)^2 + y^2)(x^2 + y^2)} \rho(x) \, dx = \frac{2\pi}{x_0} - 1.
$$
Since $x^2 \le x^2+ y^2$ and $y^2 \le x^2 + y^2$, we have
$$
0 \le \frac{1}{\pi} \int_{\R} \frac{x^2}{((x - x_0)^2 + y^2)(x^2 + y^2)} \rho(x) \, dx \le \frac{1}{\pi} \int_{\R} \frac{\rho(x) \, dx}{(x - x_0)^2 + y^2}
$$
and
$$
0 \le \frac{1}{\pi} \int_{\R} \frac{y^2}{((x - x_0)^2 + y^2)(x^2 + y^2)} \rho(x) \, dx \le \frac{1}{\pi} \int_{\R} \frac{\rho(x) \, dx}{(x - x_0)^2 + y^2}.
$$
In the last integral we can pass to the limit, say due to the Lebesgue theorem,
$$
\lim_{y \to \infty} \frac{1}{\pi} \int_{\R} \frac{\rho(x) \, dx}{(x - x_0)^2 + y^2} = \frac{1}{\pi} \int_{\R} \lim_{y \to \infty} \frac{\rho(x) \, dx}{(x - x_0)^2 + y^2} = 0.
$$
Thus, $2\pi/x_0 - 1 = 0$.
\end{proof}

Note that $\Theta_2$ in Lemma~\ref{l.periodictheta} is uniquely defined only modulo an additive real constant. We fix a unique such function by requiring in addition that $\Theta_2(0) = 0$. Subject to this condition, the comb domain $\Pi_2$ and the map $\Theta_2$ are in 1-1 correspondence. We can define a set $E_2$ of the form \eqref{e.e2def} as follows. The definition will depend on another parameter $e^{i\psi_0} \in \T$. Let
$$
E_2 = E_2(\Pi_2,e^{i\psi_0}) = \{ e^{i(\psi + \psi_0)} : \psi \in \Theta_2^{-1}(\R) \}.
$$

The fact that $\Theta_2$ is $2\pi$-periodic allows us to define a function $f$ on $\bar \D$ by
$$
f(\varphi) = f \big( e^{i(\psi + \psi_0)} \big) = e^{i \Theta_2(\psi)}.
$$
We claim that this function can be extended analytically as a multi-valued function in $\mathcal{D}$. Moreover, this function has a simple relation to our complex Green functions $b_i$ and  $b_{-i}$.

\begin{lemma}
$f$ can be extended to $\mathcal{D}$ as a multi-valued function. In other words, $f \circ \mathfrak{v}$ is a character-automorphic function on $\D$. Moreover,
$$
f \circ \mathfrak{v} = \frac{b_i b_{-i}}{b_i(t_0) b_{-i}(t_0)},
$$
where $\mathfrak{v}(t_0) = e^{i \psi_0}$.
\end{lemma}

\begin{proof}
$f$ can be extended analytically through the gap $(\varphi_0^-,\varphi_0^+)$ of $E_2$ by the symmetry principle since $\Theta_2 + \bar \Theta_2 = 2 \omega_0$ on the image of this gap. Note that by definition, $f$ has a simple zero at $0$. Therefore, it also has a simple zero at $\infty$. Moreover, $f$ has no other zeros and we have $|f \circ \mathfrak{v}| = 1$ on $\T$. This shows that $f \circ \mathfrak{v} = b_i b_{-i} e^{i \psi_2}$. To determine the unimodular constant $e^{i \psi_2}$, we note that for $t_0$ with $\mathfrak{v}(t_0) = e^{i \psi_0}$, we have
$$
b_i(t_0) b_{-i}(t_0) e^{i \psi_2} = f(e^{i \psi_0}) = e^{i \Theta_2(0)} = 1
$$
by our normalization condition. This completes the proof.
\end{proof}

\begin{coro}
We have $\mu_i(\gamma_k) = e^{i(\omega_k - \omega_0)}$.
\end{coro}

\begin{proof}
Fix a contour $(\gamma_{\mathcal{D}})_k$ corresponding to the generator $\gamma_k$ of the group $\Gamma$. We extend $f$ analytically along this contour. Due to the reflection property, we obtain
$$
(f \circ \mathfrak{v})(\gamma_k) = e^{2i(\omega_k - \omega_0)}.
$$
On the other hand, by the properties of the functions $b_i$, $b_{-i}$ this is $\mu_i(\gamma_k)^2$.
\end{proof}

To formulate our main result we need one more unimodular constant. From the representation \eqref{e.bibminusi} we get
\begin{equation}\label{e.psionedef}
e^{i \psi_1} = e^{i \psi_0} \frac{b_{-i}(t_0)}{b_i(t_0)}.
\end{equation}

\begin{theorem}\label{t.maincmv}
Consider $e^{i \psi_0} \in \T$, a periodic comb domain $\Pi_2 = \Pi_2(\{\omega_k,h_k\}_{k=1}^\infty)$ as in \eqref{e.Thetadefinition2} satisfying \eqref{e.discreteomegaks2} and \eqref{e.widomcondition2}, and the associated set $E_2 = E_2(\Pi_2,e^{i \psi_0}) \subset \T$. With $e^{i\psi_1} \in \T$ from \eqref{e.psionedef}, we assume in addition that
\begin{align}\label{e.frequenciesindependent2}
& \{ n_k \}_{k \in \Z_+} \subseteq \Z, \; \exists k_0 : \forall k \ge k_0, \; n_k = 0, \; \text{and} \\
\nonumber & \exp \Big( i \Big( n_0 \psi_1 + \sum_{k \ge 1} n_k (\omega_k - \omega_0 \Big) \Big) = 1\Rightarrow \forall k \in \Z_+, \; n_k = 0 .
\end{align}
Then we have the following dichotomy:
\begin{itemize}

\item[{\rm (a)}] If $E_2$ satisfies DCT, then every $\mathcal{E} \in \mathcal{E}(E_2)$ is almost periodic.

\item[{\rm (b)}] If $E_2$ does not satisfy DCT, then every $\mathcal{E} \in \mathcal{E}(E_2)$ is not almost periodic.

\end{itemize}
\end{theorem}

The proof of Theorem~\ref{t.maincmv} goes along the same lines as the proof of Theorem~\ref{t.main}. The basis for the proof is the following result by Peherstorfer and Yuditskii from \cite{PY06}.

\begin{theorem}[a reformulation of Theorem~1.5 in  \cite{PY06}]
There is a bijection between $\{ H^2(\alpha) \} \times \T$ and $D(E_2)$. Since the latter is homeomorphic to $\mathcal{E}(E_2)$, this gives rise to a continuous generalized Abel map $\pi_2 : \mathcal{E}(E_2) \to \Gamma^* \times \T$. The push-forward of the shift operation on the sequence of Verblunsky coefficients of a matrix in $\mathcal{E}(E_2)$ by the generalized Abel map is the translation by $(\mu_i, e^{i \psi_1})$.
\end{theorem}

The theorem is not stated in this form in \cite{PY06}, but it may be extracted from that paper. First of all, \cite[Theorem~1.5]{PY06} assumes that $E_2$ is homogeneous, but this assumption is actually not necessary for the modified statement given above. Homogeneity implies that, in our notation, $\mathcal{NTF} = \emptyset$, and hence one gets a bijection between $\Gamma^* \times \T$ and $\mathcal{E}(E_2)$ in this case (\cite[Theorem~1.5]{PY06} is formulated in such a way that $\Gamma^* \times \T$ is mapped onto $\mathcal{E}(E_2)$). Of course, one does not have such a bijection when $\mathcal{NTF} \not= \emptyset$. Secondly, the second component of $(\mu_i, e^{i \psi_1})$ is not described explicitly there but it may be identified as $e^{i \psi_1}$ given our discussion above, see also  the second remark in \cite[Section~4.3]{PY06}.

Once Theorem~\ref{t.maincmv} is obtained, one shows that the second alternative is possible and hence derives Theorem~\ref{tecmvklcfails} in the same way Corollary~\ref{c.main} was derived above.

Concerning an explicit formula for the Verblunsky coefficients in terms of the reproducing kernels $k^\alpha(\zeta,\kappa)$ and $k^\alpha(\zeta,\bar \kappa)$ and the related orthonormal basis in $H^2(\alpha)$, see Theorems~4.1 and 4.4 in \cite{PVY09}.

\end{appendix}

\section*{Acknowledgment}

D.\ D.\ would like to express his gratitude to the Institut f\"ur Analysis, Abteilung f\"ur Dynamische Systeme und Approximationstheorie at the Johannes Kepler University, Linz for the hospitality during a visit in the spring of 2014 where this work was done, and for financial support through the Austrian Science Fund FWF, project no: P22025-N18.

\end{document}